\documentclass[preprint,12pt]{elsarticle}
  \usepackage[final,				%
	  		  stretch = 30,		%
			  shrink	=30,		%
			  tracking = true%
			]{microtype}
  \UseMicrotypeSet[protrusion]{basicmath}

  \usepackage{amssymb}
  \usepackage{amsmath}
  \usepackage{todonotes}
    \usepackage{hyperref} 
  \usepackage{amsthm}
  
  \newtheorem{definition}{Definition}[section]
  
  \newtheorem{remark}{Remark}[section]

  \newtheorem{prop}{Proposition}[section]
    
  \newtheorem{lemma}{Lemma}[section]

  \usepackage{mathtools}
  \mathtoolsset{showonlyrefs}
  \usepackage{colortbl}
  
\definecolor{sp1}{RGB}{153,204,51}
\definecolor{sp1D}{RGB}{111, 154, 26}
\definecolor{sp1DD}{RGB}{68,103,0}
\definecolor{sp1H}{RGB}{207,227,168}
\definecolor{sp1HH}{RGB}{230,243,205}

\definecolor{sp2}{RGB}{255,150,0}
\definecolor{sp2D}{RGB}{196, 98, 0}
\definecolor{sp2DD}{RGB}{136,45,0}
\definecolor{sp2H}{RGB}{255,204,153}
\definecolor{sp2HH}{RGB}{255,242,214}

\definecolor{sp3}{RGB}{0,186,219}
\definecolor{sp3D}{RGB}{0, 132, 161} 
\definecolor{sp3DD}{RGB}{0,77,103}
\definecolor{sp3H}{RGB}{171,227,236}
\definecolor{sp3HH}{RGB}{214,241,246}

\definecolor{WtalPink}{RGB}{213, 0, 109}%
\definecolor{WtalPinkD}{RGB}{180, 0, 99}%
\definecolor{WtalPinkDD}{RGB}{144, 0, 79}
\definecolor{WtalPinkH}{RGB}{255, 93, 153}%
\definecolor{WtalPinkHH}{RGB}{255, 139, 179}%

\definecolor{W}{RGB}{161, 118, 178}
\definecolor{WtalLavendel}{RGB}{161, 118, 178}
\definecolor{WtalLavendelD}{RGB}{127, 78, 156}
\definecolor{WtalLavendelDD}{RGB}{102, 62, 125}
\definecolor{WtalLavendelH}{RGB}{205, 164, 222}
\definecolor{WtalLavendelHH}{RGB}{226, 200, 240}%

\definecolor{sps}{RGB}{201,218,228}
\definecolor{WtalGrauBlau}{RGB}{124,151,176}
\definecolor{WtalGrauBlauD}{RGB}{71,87,102}
\definecolor{WtalGrauBlauDD}{RGB}{57, 70, 82}
\definecolor{WtalGrauBlauH}{RGB}{201,218,228}
\definecolor{WtalGrauBlauHH}{RGB}{233,242,247}

\definecolor{strasse}{RGB}{174, 174, 190}
\definecolor{strasseLinie}{RGB}{255, 255, 255}

  \usepackage{caption}
  \usepackage{subcaption}

  \usepackage{tikz}
  \usepackage{xcolor}
  \DeclareMathOperator{\WSoperator}{WS}
  
  \DeclareMathOperator{\Rectangle}{Rect}
  \DeclareMathOperator{\unif}{DiscreteUniform}
  \DeclareMathOperator{\HVoperator}{HV}
  \DeclareMathOperator{\HVRoperator}{HVR}
  \DeclareMathOperator{\PROPoperator}{PROP}
  \usepackage{tabularx}
  \usepackage{booktabs}

\newif\iflong
\longfalse

  \newlength{\mywidth}%
  \makeatletter
    \newcommand\myfbox[2][\linewidth]{%
    \xdef\mysep{\dimexpr 1\dimexpr\f@size pt\relax\relax}%
    \setlength{\fboxsep}{\mysep}%
    \setlength{\mywidth}{\dimexpr#1-2\fboxsep-2\fboxrule\relax}%
    \noindent\fbox{\begin{minipage}[inner sep=0]{\mywidth}#2\end{minipage}}%
    }%
\makeatother

  \usepackage[linesnumbered,vlined,ruled,commentsnumbered]{algorithm2e}

  \journal{Preprint ArXiv}

\begin{document}

  \begin{frontmatter}

    \title{A multi–objective perspective on block–structured integer programs with one soft coupling constraint}

  \author[label1]{Mark Lyngesen}
  \affiliation[label1]{organization={Aarhus University, Department of Economics and Business Economics},
              addressline={Universitetsbyen 51},
              city={Aarhus C},
              postcode={DK-8000},
              country={Denmark}}
  \affiliation[core]{organization={Center for Research in Energy: Economics and Markets},
              addressline={Universitetsbyen 51},
              city={Aarhus C},
              postcode={DK-8000},
              country={Denmark}}
  \author[label2]{Kathrin Klamroth}
  \author[label2]{Britta Efkes}
  \author[label1,core]{Sune Lauth Gadegaard}
  \affiliation[label2]{organization={University of Wuppertal},
              city={Wuppertal},
              country={Germany}}

  \begin{abstract}

  This paper presents a multi-objective perspective on block-structured integer 
programs featuring a single soft coupling constraint. By interpreting the 
coupling constraint as a second objective, we transform the coupled single-objective 
problem into an additively-separable bi-objective optimization problem. 
To avoid the expensive computation of the full Pareto front, we introduce an algorithm, which uses a binary search to isolate a region of interest around the soft constraint limit. This algorithm provides provable bounds on the single-objective optimum. 
We further enhance this algorithm, by exploiting the block-structure, using a novel $\lambda$-lookup mechanism to skip repeated sub-problem calculations. Finally, for scenarios requiring all non-dominated solutions within the region of interest, we propose a new approach, that works its way from the middle of the region of interest outwards. This algorithm shows quick convergence in terms of representation.
Computational studies demonstrate that our methods dramatically reduce integer programming calls, thereby outperforming traditional dichotomic search. For large instances the method works as a strong heuristic providing bounds on the gap to an optimal solution, providing trade-off information in addition to the solution.
  \end{abstract}


  \end{frontmatter}

  \newcommand{\Y}{\mathcal{Y}} %
  \newcommand{\Yl}{\mathcal{Y}_\lambda} %
  \newcommand{\Ylh}{\mathcal{Y}_{\hat\lambda}} %
  \newcommand{\fl}{f_\lambda} %
  \newcommand{\Yn}{\mathcal{Y}_{\texttt{N}}} %
  \newcommand{\nondom}[1]{\left(#1\right)_{\texttt{N}}} %
  \newcommand{\Yr}{\mathcal{Y}_{R}} %
  \newcommand{\hYr}{\hat{\mathcal{Y}}_{R}} %
  \newcommand{\Yse}{\mathcal{Y}_{\texttt{SE}}} %
   \newcommand{\Ys}{\mathcal{Y}_{\texttt{S}}} %
  \newcommand{\Xe}{\mathcal{X}_{\texttt{E}}} %
  \newcommand{\Xse}{\mathcal{X}_{\texttt{SE}}} %
  \newcommand{\Xs}{\mathcal{X}_{\texttt{S}}} %
  \newcommand{\Ysup}{\mathcal{Y}_{\texttt{SUP}}} %
  \newcommand{\X}{\mathcal{X}} %
  \renewcommand{\S}{\mathcal{S}} %
  \newcommand{\wsd}{\tilde{\mathcal{W}}^0} %
  
  \newcommand{\N}{\texttt{N}} %
  \newcommand{\WS}[1]{\WSoperator(#1)}
  \newcommand{\WSs}[2]{\WSoperator(#1)\^{#2}}
  \newcommand{\aWS}[1]{\arg\WSoperator(#1)}
  \newcommand{\aWSs}[2]{\arg\WSoperator(#1)\^#2}
    \newcommand{\lfrac}[2]{(#1)/(#2)} %

  \renewcommand{\^}[1]{^{(#1)}} %
  \newcommand{\yb}{\bar{y}} %

  \newcommand{\Rpp}{\mathbb{R}^p_\geqq} %
  \newcommand{\R}{\mathbb{R}} %
  \newcommand{\Z}{\mathbb{Z}} %

  \newcommand{\msum}{\bigoplus} %

  \newcommand{\nd}{ND\@\xspace} %
  \newcommand{\ND}{nondominated\@\xspace} %
  \newcommand{\MS}{Minkowski sum\@\xspace} %
  \newcommand{\SP}{sub-problem\@\xspace} %
  \newcommand{\eg}{e.g.\@\xspace} %
  \newcommand{\eeg}{E.g.\@\xspace} %
  \newcommand{\wrt}{wrt.\@\xspace} %
  \newcommand{\ie}{i.e.\@\xspace} %
  \newcommand{\BB}{branch-and-bound\@\xspace} %

  \newcommand{\bsip}{BSIP\@\xspace}
  \newcommand{\softconstraint}{soft constraint\@\xspace}
  \newcommand{\nfold}{$N-$fold\@\xspace}
  \newcommand{\BSIP}{Block structured integer programs\@\xspace}
  \newcommand{\roi}{ROI\@\xspace}
  \newcommand{\ROI}{Region of Interest\@\xspace}
  \newcommand{\mo}{MO\@\xspace}
  \newcommand{\MO}{multi-objective\@\xspace}
  \newcommand{\dm}{DM\@\xspace}
  \newcommand{\DM}{decision maker\@\xspace}
  \newcommand{\RHS}{right-hand-side\@\xspace}
  \newcommand{\LHS}{left-hand-side\@\xspace}

  \newcommand{\phaseone}{\texttt{Phase-1 method}\@\xspace}
  \newcommand{\findRoi}{\texttt{Find ROI}\@\xspace}
  \newcommand{\findRoiDec}{\texttt{Find ROI Decomposed}\@\xspace}
  \newcommand{\findRoiDecLook}{\texttt{Find ROI Decomposed Lookup}\@\xspace}
  \newcommand{\LookupLambda}{\texttt{LookupLambda}$(\Lambda\^s,\lambda)$\@\xspace}
\newcommand{\solveRoi}{\texttt{Solve ROI}\@\xspace}
  \renewcommand{\l}{\lambda} %

\pagebreak
\section{Introduction}
  Many optimisation problems addressing real world problems are expressed using \emph{hard} constraints. By that we mean constraints that \emph{must} be satisfied by every potential solution. However, it is often so that constraints are not \emph{hard} and that they can be violated by some amount, if the resulting solution quality is sufficiently improved by allowing the violation. This leads to the notion of \emph{soft} constraints that \emph{ideally} should be satisfied, but may be violated if doing so is sufficiently beneficial. In case of soft constraints, it is important for the decision maker to understand the trade-off's between the violation and the solution quality.

  In this paper, we study a particular kind of optimisation problem with a single, soft constraint, namely block structured integer programming problems with a single soft, coupling constraint. The general idea behind this study is to exploit the problem structure by treating the soft coupling constraint as an additional objective function instead of as a constraint. This allows us to utilise the block structure of the problem by decomposing it into several independent sub-problems. We then use multi-objective techniques to analyse the trade-offs between the original objective function and the soft coupling constraint, and to eventually provide an optimal solution along with valuable trade-off information for the decision maker.  
  
  In the remainder of this section we will review the related literature and emphasise the contributions of the paper.

  Block-structured integer programs often occur in integer and combinatorial optimisation. There are several types of block-structures and we refer to \citet{chen2019block} for an overview of many of these. %
  In this paper we consider constraint matrices with an \emph{\nfold} block structure. These are matrices consisting of blocks of sub-matrices on the diagonal and additional coupling constraints. As in \citet{eisenbrand2018faster} %
  we consider a \emph{generalised} \nfold matrix where the sub-matrices/blocks are not necessarily identical, and similarly, the blocks defining the coupling constraints are allowed to vary.
  Problems with an \nfold block structure have applications in several areas (see \citet{knop2020,chen2019block,cslovjecsek2025}) and contain the structure used to model a variety of integer optimization problems. 

  Many other approaches utilize block structured constraint matrices by decomposing the problem into master- and sub-problems. Two such especially well studied approaches are Benders' decomposition \citep{benders1962partitioning} and Dantzig-Wolfe decomposition \citep{desrosiers2024branch,dantzig1960decomposition}. We also refer the interested reader to the textbook \cite{martin2012large} for a unified view on the two decomposition approaches.
  
  As mentioned above, this paper considers IPs with an \nfold block structured matrix with a single coupling constraints. Such structures occur when optimizing over multiple independent subsystems (not to be confused with independence systems). The coupling constraint could refer to a limited resource shared among the subsystems, for example a decision maker wanting to maximize overall profit for a set of activities each consuming some amount of CO\textsubscript{2}, subject to a maximum bound on emissions. We propose a \emph{multi-objectivization} approach, were the soft constraint is turned into a second objective function to be optimized together with the original objective function of the problem.

  The term multi-objectivization refers to the general idea of solving single-objective optimization problems using multi-objective optimization methods. The paper \citet{segura2016using} reviews the concept of multi-objectivization for evolutionary algorithms. The review concludes that the methods generally perform \emph{worse} compared to single-objective solvers. However, the authors note that one can use \MO evolutionary algorithms to maintain diversity of solutions. There are also instances where multi-objectivization schemes outperform single-objective schemes, see again \citet{segura2016using}. The reformulation into multi-objective problems in some cases allows transforming a single-objective constrained problem to a multi-objective unconstrained problem — or at least into a problem with fewer constraints. There are also applications of multi-objectivization techniques in multiplicative programming \citep{shao2014objective,shao2016primal} where the single-objective problem is solved as a multi-objective problem, and in optimization under uncertainty where each scenario is used to define an individual objective \cite{klamroth2013uncertain}.

  Another use of multi-objectivization for multi-dimensional knapsack problems is that of \citet{schulze2017soft} (see also \citet{schulze2017new}), where the authors interpret a soft constraint as an objective function. By interpreting the \LHS of the constraint as an objective, the problem is turned into a bi-objective problem. The new problem is then solved using a dynamic programming algorithm, and the second objective in turn provides the decision maker with a variety of \emph{interesting} solutions \emph{around} the original right-hand-side of the soft constraint. 
  
  When modelling multi-objective problems that result, e.g., from re-interpreting a soft constraint as an additional objective function, one can make use of preference information of the decision maker in several ways. This process is referred to as preference-driven multi-objective optimization \cite{wang2017}. See \citet{rostami2017progressive} for a discussion of the incorporation of decision maker preferences into multi-objective problems. An \emph{a priori} approach would be for the decision maker to provide enough information on preferences to construct a problem with a single unique optimal value. Such information could be substitution weights between objectives, or a target value vector where the closest non-dominated point is chosen. In the typical \emph{a posteriori} approach seen in multi-objective research, a decision maker is presented with all non-dominated alternatives between objectives. %
  \citet{rostami2017progressive} considers a hybrid between these two approaches, in which some preference information is incorporated into the problem, where the decision maker is interested in a set of \emph{interesting} non-dominated solutions.
  The \emph{\ROI} (\roi) of a multi-objective problem is then the area of the objective space in which a decision-maker is interested in knowing all non-dominated solutions.  
  There are different approaches for defining a \roi of a \MO problem in the literature, with the two main goals of reducing the search area and only providing the decision maker with relevant trade-off information:
  \citet{rostami2017progressive} define a \roi using a point dominated by an \emph{a priori} known \emph{preference vector}, and propose evolutionary algorithms, which find multiple solutions in the \roi. \citet{yu2025towards} define several \roi's from a set of search directions in the objective space (one for each decision maker). Then, each decision maker seeks solutions in an \roi \emph{around} their search direction,  which in turn defines several regions of interest. \citet{zhou2023efficient} define regions of interest using a reference point. %

  An intuitive way of defining the \roi, is to look for Pareto optimal solutions \emph{around} the soft constraint limit. That is, solutions for which the left-hand-side of the soft constraint is just over, and just under the limit $W$. The decision maker is allowed to provide a parameter $\gamma \in [0,1]$ defining the size of the \roi, with $1$ being the entire Pareto front and $0$ being the smallest part of the Pareto front containing supported points around the soft constraint limit. See \autoref{fig:visual-ROI} for a visualization of the \roi for different levels of $\gamma$. %
  In \citet{schulze2017new} the \roi for $\gamma=0$ is defined as the rectangle with corner points defined by the two extreme supported points $y^+, y^-$ immediately above and below the soft constraint value $W$.%

  \begin{figure}[h!]
    \begin{center}
      \includegraphics[width=0.8\textwidth]{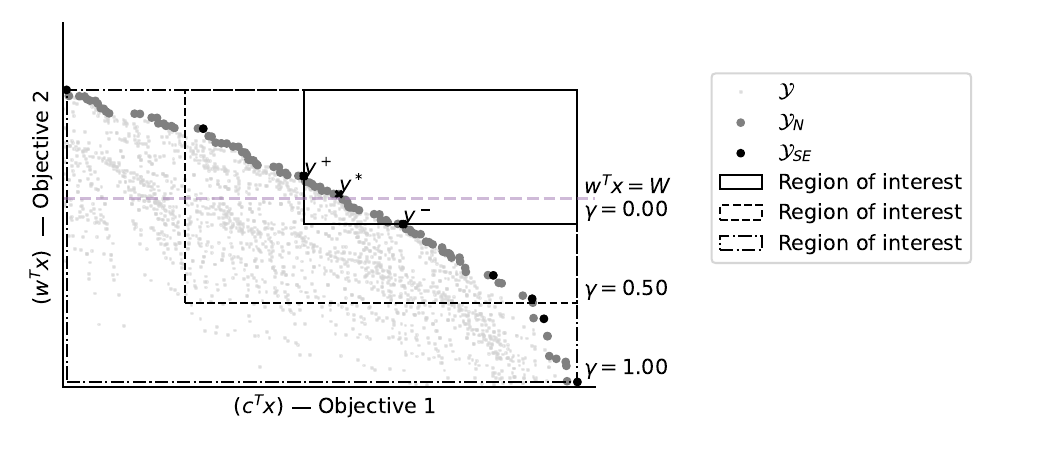}
    \end{center}

    \caption{A visualisation of the region of interest for different levels of $\gamma \in [0,1]$
    }\label{fig:visual-ROI}
  \end{figure}

  As we will see, the problem resulting from the multi-objectivization can be characterised as an additively-separable multi-objective problem. Such problems are decomposable and have been studied in \citet{Gardenghi2011} and \citet{Kerberenes2022phd}. This present paper seeks to provide solution methods for a class of additively separable multi-objective problems, requiring only solutions in a specified \roi.

  The contribution of this paper is four-fold:
  \begin{enumerate}
      \item We propose to solve soft constrained \nfold block structured integer programs using multi-objectivization.
      \item We develop a new approach for computing the region of interest, utilising the decomposable structure of the problem.
      \item For providing trade-off information around the soft constraint in the region of interest, we develop an ``alternating direction''-enhanced $\epsilon$-constraint approach for solving the bi-objective problem inside the region of interest.
      \item Through extensive computational results we analyse the effectiveness of our approaches.
  \end{enumerate}
  
  The remainder of the paper is organised as follows:
\autoref{sec:prereq} presents preliminaries  on \MO optimization concepts and solution methods, as well as a full formulation of the block-structured integer program (\bsip) considered in this paper.
  The theoretical contributions of the paper are presented in \autoref{sec:method}: First we explain the methodology of the paper, how we multi-objectivize the \bsip problem, and how this leads to decomposing the problem.
  In \autoref{sec:theory-upper-bound} we provide bounds on the number of extreme supported points in the multi-objective counterpart problem. %
  \autoref{sec:theory-find-roi} presents an algorithm for finding the two points defining the \ROI, along with a way of calculating the optimality gap of the resulting solutions. In \autoref{sec:theory-solve-roi} we present a decomposition algorithm for finding all points in the \ROI.
  Finally, in \autoref{sec:comp-study} we present a computational study on a testset of \bsip instances.

\section{Preliminaries}\label{sec:prereq}

  \subsection{Block-structured integer programming problems}
  We consider single objective block-structured integer programming problems (\bsip) which are IPs with a block-structured constraint matrix. We assume blocks are coupled by a single coupling constraint. In particular, the global problem can be described as an \nfold IP where each block may be different from one another as in \citet{eisenbrand2018faster}.

  Let $A\in \Z^{m\times n}$ be an \nfold block matrix with $S$ blocks, indexed by the set $\S =\{1, \ldots, S\}$, such that block $A\^{s}$ consists of $m\^s$ rows and $n\^s$ columns for each $s \in \S$, and let $b\in\Z^m$. Throughout this paper, we shall assume that the set $\X\coloneq\{x \in \Z^n \mid Ax\le b\}$ is bounded. This is true for any IP with bounds on the integer variables. Moreover, $wx \ge W$ (where $wx$ denotes the inner product of $w$ and $x$) denotes the coupling constraint, with $w\in\Z^n$ and $W\in\Z$. 
  The general formulation of a \emph{block-structured integer programming problem} \eqref{eq:bs-ip-formulation} with a single \emph{soft constraint} is given by
  \begin{align}
  \tag{BSIP}\label{eq:bs-ip-formulation}
  \qquad \qquad \max\;  & cx \\
  \text{s.t. }& Ax \leqq b, &\\
        & wx \ge W, & \text{(soft constraint)}
        \\
        & x \in \Z^n. & 
  \end{align}
  We assume throughout this paper that problem \eqref{eq:bs-ip-formulation} is feasible. 
  Let $\X^*$ denote the set of optimal solutions for problem \eqref{eq:bs-ip-formulation}. For any $x^* \in \X^*$ the optimal objective value is $cx^*$ while the value of the soft constraint is $wx^*$. Ideally, $wx^*$ is as large as possible.

    To get an intuitive understanding of the problem we investigate, we may represent the problem in the standard form 
    as follows:
    \begin{equation}
      \max \{cx \mid \bar{A}x \leqq \bar{b}, x\in \mathbb{Z}^n\}
    \end{equation}
    where $\bar{A}$ and $\bar{b}$ are obtained by appending $-w$ as the last row of $A$ and $-W$ to $b$. A visualization of this is presented in \autoref{table:visual-instances}.

    \begin{figure}
        \begin{center}
            \begin{tabular}{cccccccc}
                           & & & &&&\\
                    \cellcolor{white} $\max$ & \cellcolor{white}$ \cellcolor{sp1}c^{(1)}$ & \cellcolor{sp2} $c^{(2)}$ & \cellcolor{sps}$\cdots$ & \cellcolor{sp3} $c^{(S)}$&& \\
                    \midrule
                    s.t.& $A^{(1)}$ \cellcolor{sp1}&&  &&& $b^{(1)}$ \cellcolor{sp1}\\
                        & & $A^{(2)}$ \cellcolor{sp2}& &&& $b^{(2)}$ \cellcolor{sp2}\\
                           & & & \cellcolor{sps}$\ddots$ &&& $\vdots$ \cellcolor{sps}\\
                           & & &  &$A^{(S)}$ \cellcolor{sp3}&&$b^{(S)} $ \cellcolor{sp3}\\
                            \cellcolor{white} & \cellcolor{sp1} $-w^{(1)}$ & \cellcolor{sp2}$-w^{(2)}$ & \cellcolor{sps}$\cdots$ & \cellcolor{sp3}$-w^{(S)}$ & & $-W$ \cellcolor{W}\\
                    \midrule
            \end{tabular}\caption{Illustration showing the \nfold block structure of \eqref{eq:bs-ip-formulation} with $N=S$ blocks.
            }\label{table:visual-instances}
                \end{center}
    \end{figure}

  \subsection{Multi-objective optimization}

    In this study, we use the binary relations $<$, $\leqq$, and $\leq$ when comparing vectors in $\mathbb{R}^p$. For two vectors $y^1,y^2\in\mathbb{R}^p$ we have
  \begin{align}
      y^1<y^2&\Leftrightarrow y^1_k<y^2_k,\text{ for }k=1,...,p\\
      y^1\leqq y^2&\Leftrightarrow y^1_k\leq y^2_k,\text{ for }k=1,...,p\\
      y^1\leq  y^2&\Leftrightarrow y^1\leqq y^2 \text{ and }y^1\neq y^2 .
  \end{align}
  Furthermore, for two vectors $y^1$ and $y^2$ in $\mathbb{R}^p$ we define the lexicographic order as follows: if $y^1\neq y^2$ let $k^*=\min\{k\mid y_k^1\neq y_k^2\}$. We then say that $y^1\leq_{\text{lex}}y^2$ if $y^1=y^2$ or $y^1_{k^*}<y^2_{k^*}$. Furthermore, we say that $y^1<_{\text{lex}}y^2$ if $y^1_{k^*}<y^2_{k^*}$.
  
  For a multi-objective optimisation problem of the form
  \begin{equation}  
    \max\{ Cx: Ax\leqq b, x\in \Z^n\} \label{eq:mult_opt}
  \end{equation}
  where $C\in\Z^{p\times n}$, $A\in\Z^{m\times n}$, $b\in\Z^m$, and $\X=\{x\in\Z^n\mid Ax\leqq b\}$, we say that a solution, $\bar{x}\in\X$ is an \emph{efficient solution} if there is no other feasible solution $x\in\X$ such that $Cx\ge C\bar x$. The set of all efficient solutions will be denoted by $\Xe$. The image of an efficient solution $\bar x\in\Xe$, $y=C\bar x$, is called a \emph{non-dominated outcome vector} (or \emph{non-dominated point}). The set of all non-dominated outcome vectors is denoted by $\Yn$.

  We will further divide the set of efficient solutions and non-dominated outcomes into those that are supported and those that are unsupported. Following the notation in \citet{ehrgott2005multicriteria} (see also \citet{koenen2025supported} for a more detailed analysis that is particularly relevant when $p>2$), we say that $\hat{x}\in\Xe$ is a \emph{supported efficient solution} if there exists some $\tilde\lambda\in\mathbb{R}^p$, with $\tilde\lambda>0$, such that $\hat{x}$ is an optimal solution to the weighted sum problem $\max\{\tilde\lambda Cx\mid x\in\mathcal{X}\}$.
  The set of all supported efficient solutions and their corresponding supported non-dominated vectors are denoted $\Xs$ and $\Ys$, respectively. We will later focus on bi-objective problems, \ie, the case that $p=2$. Then, solutions in $\Xe\setminus\Xs$ are called \emph{unsupported efficient solutions} and their images are \emph{unsupported non-dominated outcome vectors}.

  It is well known that solving weighted sum scalarisations of the multi-objective optimisation problem is guaranteed to lead to supported efficient solutions provided the weight vector is strictly positive. In addition, by varying the weight vector, all supported efficient solutions can be found (see, \eg, \cite{ehrgott2005multicriteria}).
  
  Instead of considering $p-$dimensional weight vectors $\tilde\lambda \in \R^p$ we shall consider weight vectors $\lambda$ in the $(p-1)-$dimensional \emph{normalised} weight set 
  $\wsd \coloneq\{\lambda\in\mathbb{R}^{p-1}\mid\lambda>0,\sum_{l=1}^{p-1}\lambda_l<1\}$. Note that this is simply the interval $\wsd=(0,1)$ in the bi-objective case. 
  
  For the remainder of this paper we consider problems with $p=2$, and for any $\lambda \in \wsd$ we define the corresponding (normalised) weighted sum problem as:
    \begin{equation}
    \tag{$WS(\lambda)$}
  \max\{\lambda(C_1x) + (1-\lambda)(C_2x) : Ax\leqq b, x\in\Z^n\},%
    \label{prob:WS}
  \end{equation}
  where we always assume that the feasible set $\X$ is non-empty and bounded. 
  Here, $C_i$ denotes the $i$'th row in $C$. %
  In this case, the set of supported efficient solutions is discrete and finite, and we may write $\Xs=\{\hat{x}\^1,...,\hat{x}\^L\}$ and set $\hat y\^l=C\hat x\^l$, for $l=1,...,L$.
  The weight set $\wsd$ can be decomposed into subsets $\wsd(\hat{y}\^1),...,\wsd(\hat{y}\^L)$, such that for each $l =1,\ldots, L$, $\hat{x}\^l$ is an optimal solution to \eqref{prob:WS} for any $\lambda \in \wsd(\hat{y}\^l)$.
  We say that $\{\wsd(\hat{y}\^l)\}_{l=1}^L$ is a \emph{weight set decomposition}, and note that $\bigcup_{l=1}^L\wsd(\hat{y}\^l) = (0,1)$.

  Note that some of the \emph{weight cells} $\wsd(\hat{y}\^l)$ may have dimension $0$, \ie, $\hat{x}\^l$ is optimal for exactly one weighting vector $\lambda\in\wsd$. In the following, we will focus on weight cells $\wsd(\hat{y}\^l)$ that have dimension $1$ and call the associated non-dominated outcome vectors \emph{extreme supported}. Their corresponding extreme supported non-dominated vectors are denoted $\Xse$ and $\Yse$, respectively. 
  We refer to \citet{Przybylski2010} for a more detailed description of the weight set decomposition.

  \subsubsection{Computing a weight set decomposition}
  For multi-objective optimisation problems with an arbitrary number of objectives it is generally not a trivial task to compute a weight set decomposition. The interested reader is referred to \citet{benson2002weight, przybylski2010recursive, bokler2015output} and \citet{halffmann2020inner} for a thorough treatment of the computation of the weight set decomposition for weighted sum scalarisations and to \citet{helfrich2024weighted} for the analysis of weight set decompositions for general weighted $p$-norms.

  For the bi-objective case, where $p=2$, the weight set decomposition can be computed using \emph{dichotomic search}, originally, and independently, proposed by \citet{cohon2013multiobjective} and \citet{aneja1979bicriteria}. The method is also referred to as the ``Phase-1 method'', as it is often used in localisation methods for bi-objective integer programs, as a first phase for finding the extreme supported non-dominated outcome vectors. The dichotomic search method starts by computing the two lexicographic optima. We say that a feasible solution $x^*$ is lexicographically optimal if there does not exist another feasible solution $x'$ such that $Cx^*<_{\text{lex}}Cx'$. The normal to the line connecting the images of the two lexicographic optima in objective space is then used as a weight vector for the weighted sum problem. If the weighted sum problem deems the lexicographic solutions optimal, the search stops. Otherwise, a new supported efficient solution is found. The directions defined by the normals between the newly found point and the existing images are then searched in a similar manner. This continues until no new extreme supported solutions are found. The procedure can be summarized as in \autoref{ALG:PhaseOne}.  %
  If \eqref{prob:WS} has several optimal solutions for a given weight $\lambda$ (\ie extreme and non-extreme supported solutions), then a solver could return any of these. To ensure that only the required extreme supported solutions are returned, we make use of \autoref{remark:onlyExtremeLambda}.

\begin{remark}\label{remark:onlyExtremeLambda}  By adding a sufficiently small positive value to each tested weight $\lambda$ in \autoref{ALG:PhaseOne}, we can avoid finding non-extreme supported non-dominated outcome vectors. Indeed, non-extreme supported non-dominated outcome vectors need to lie on lines with specific slopes since all outcome vectors are integral, and we make sure such slopes are avoided. 
\end{remark}

    \begin{algorithm}[h!]
            \KwData{Constraint matrix $A$, vector $b$, and cost-matrix $C = \begin{pmatrix}
                C_1 \\ C_2
            \end{pmatrix}$.}%
            \KwOut{Sets $\Yse$ and $\Xse$}
            \BlankLine
            \tcc{Compute lex-max solutions}
            $x^{lr}\in\arg\mbox{lex}\max\{(C_2x,C_1x)\mid x\in\mathcal{X}\}$\;
            $x^{lr}\in\arg\mbox{lex}\max\{(C_1x,C_2x)\mid x\in\mathcal{X}\}$\;%
            $y^{ul}\gets(C_1x^{ul},C_2x^{ul})$, $y^{lr}\gets(C_1x^{lr},C_2x^{lr})$\;
            \If{$y^{ul}=y^{lr}$}{
                \Return $\{y^{lr}\}$, $\{x^{lr}\}$\tcc*{Only one non-dominated point}
            }
            \BlankLine
            \tcc{Initialize sets and pointers}
            $y^+ \gets y^{ul}$ and $y^- \gets y^{lr}$\;
            $\Yse\gets \{y^{ul},y^{lr}\}$\tcc*{Initialise $\Yse$}
            $\Xse\gets \{y^{ul},x^{lr}\}$\tcc*{Initialise set of pre-images}
            \BlankLine
            \tcc{Enter main loop}
            \While{$y^+\neq y^{lr}$}{
                $\lambda\gets \tfrac{y_2^+-y_2^-}{(y_1^--y_1^+)+(y_2^+-y_2^-)}$\tcc*{Update weight parameter}
                $x^*\in\arg\max\{(\lambda C_1+(1-\lambda)C_2)x\mid x\in\mathcal{X}\}$\;
                $y^*= Cx^*$\;%
                \If{$(\lambda C_1+(1-\lambda)C_2)x^*>\lambda y^+_1+(1-\lambda)y_2^+$}{
                    \tcc{New solution found, update sets}
                    $\Xse\gets \Xse\cup\{x^*\}$\;
                    Insert $y^*$ into $\Yse$ between $y^+$ and $y^-$\;
                }
                \Else{
                    $y^+\gets y^-$\tcc*{No new solution found, move on}
                }
                Set $y^-$ equal to the point to the right of $y^+$ in $\Yse$
            }   
            
            \caption{\phaseone}\label{ALG:PhaseOne}
  \end{algorithm}
  \autoref{ALG:PhaseOne} works in the \emph{normalized} weight set, and computes the set of extreme supported outcome vectors. Hence, it implicitly also computes the weight set decomposition, \ie, the set of all weight cells of dimension $1$.

\section{Methodology}\label{sec:method}

  This section contains the theoretical contributions of the paper. 
  In Subsection~\ref{sec:bs-mo} we describe how the problem \eqref{eq:bs-ip-formulation} can be multi-objectivised into an additively separable multi-objective problem, which in turn can be decomposed into several sub-problems.
  In Subsection~\ref{sec:theory-upper-bound} we provide an upper bound on the number of extreme supported points defined by the number of extreme supported points in the sub-problems.
  Then, we describe the theoretical argument behind the so-called $\lambda$-lookup, and present an algorithm based on this in Subsection~\ref{sec:theory-find-roi}. Lastly, in Subsection~\ref{sec:theory-solve-roi} we present an algorithm for finding all points in the region of interest, an algorithm which we argue quickly converges towards the \emph{interesting} set.

\subsection{Multi-objectivization of block-structered problems}\label{sec:bs-mo}

  First we present a way of transforming the problem formulated in \eqref{eq:bs-ip-formulation} into an additively separable multi-objective problem shown in \eqref{prob:bs-ip-mo}:
  \begin{equation}
  \Yn = \max \{ (cx, wx) \mid Ax \leqq b, %
  x\in \mathbb{Z}^n\}.\tag{$P$} \label{prob:bs-ip-mo}
  \end{equation}
  This is achieved by treating the \LHS of the soft constraint $wx \ge W$ as a second objective to be maximized and ignoring the \RHS $W$ (In the following $W$ will be used to define the region-of-interest for the multi-objective problem). 
  
  For notational convenience let $\mathcal{X}\coloneqq\{x\in \mathbb{Z}^n \mid Ax \leqq b%
  \}$ denote the set of feasible solutions to \eqref{prob:bs-ip-mo}, and $\Y = C\X$, with $C = \begin{pmatrix}
                c \\ w
            \end{pmatrix}$.
  It is clear %
  that one can find an optimal solution $x^*$ to the single-objective problem \eqref{eq:bs-ip-formulation} among the efficient solutions $\Xe$ for multi-objective problem \eqref{prob:bs-ip-mo}. 

  So far the multi-objectivization step has only made the problem harder (possibly intractable) and the \RHS $W$ is no longer considered.
  However, as we will see, the resulting multi-objective problem is highly decomposable into additively separable multi-objective problems. %

  \newcommand{\makesingletable}[2]{%
          \color{#2}
          \begin{aligned}
                  \mathcal{Y}_N^{(#1)} = \max (cx^{(#1)}&, w x^{(#1)}) \\
                  s.t.\quad A^{(#1)} x^{(#1)} &\le b^{(#1)} \\
                                  x^{(#1)} &\in \mathbb{Z}^{n^{(#1)}}
          \end{aligned}
  }

  The resulting multi-objective problem decomposes into $S$ independent sub-problems \eqref{prob:SPs} indexed by the set $\S \coloneq \{1,\ldots, S\}$:
  \begin{align}\tag{SP$\^s$}\label{prob:SPs}
                  \Yn\^s = \max (C\^sx\^s) \\
                  s.t.\quad A\^s x\^s &\le b\^s \\
                                  x\^s &\in \mathbb{Z}^{n\^s}
  \end{align}
  where $C\^s = \begin{pmatrix}
                c\^s \\ w\^s
            \end{pmatrix}$ (see \autoref{table:visual-instances}).
  The feasible set in subproblem $s$ is denoted $\X\^s$ and the corresponding feasible sets in objective space is denoted by $\Yn\^s$. %
  
  From our notation, we see that $\X = \X\^1\times \cdots \times \X\^S$ and $\Y = \msum_{s \in \S} \Y\^s$. Here, `$\msum$` denotes the Minkowski sum operator, \ie $A\msum B \coloneq \{a+b\mid a\in A, b\in B\}$ and $\msum_{i=1}^nA\^i=(((A\^1\msum A\^2) \msum \cdots) \msum A\^n)$.
  It is well-known that $\Yn = \left(\msum_{s \in \S} \Yn\^s\right)_\N$ (see \citet{Gardenghi2011} for a proof).
  
  Hence, all feasible points for \eqref{prob:bs-ip-mo} can be calculated using the feasible points of the sub-problems \eqref{prob:SPs} and the set $\Yn$ can be calculated using only the \nd sets of the sub-problems.
  Analogous to \eqref{prob:SPs}, we define the (normalised) weighted sum problem of sub-problem $s$ %
  as the problem \eqref{prob:WS} with $A = A\^s, b= b\^s, c=c\^s$ and $x = x\^s$.

\subsection{Upper bound on the number of extreme-supported points}\label{sec:theory-upper-bound}

  In this subsection we provide a bound on the number of extreme supported points of $\Yse$ for \eqref{prob:bs-ip-mo}. From \citet[Prop. 3.3]{paper1} we know that any point of $\Yse$ is the sum of extreme supported points of the sub-problems. Therefore, one can bound the cardinality of $\Yse$ by $\prod_{s=1}^S |\Yse\^s|$. In the following we show that the cardinality of $\Yse$ can be bounded by the sum of the cardinality of extreme supported points over all sub-problems. The proof of this uses so-called critical weights. These are weights for which two extreme points are optimal solutions for \eqref{prob:WS}.

  \newcommand{\C}{\mathcal{C}}
  \begin{definition}
    Given a weight set decomposition $\wsd(\Yse) := \{\wsd(y) \mid y \in \Yse\}$, let $\C(\Yse) = (0,1)\setminus \bigcup_{y \in \Y} \operatorname{int}(\wsd(y)) =\{\lambda_1, \ldots, \lambda_k\}$ denote the set of critical weights. The sets $\operatorname{int}(\wsd(y))$ for $y \in \Yse$ are called the cells of the weight set decomposition.
  \end{definition}

  By definition, there is a one-to-one correspondence between the cells of $\wsd(\Yse)$ and the extreme supported points $\Yse$. For any $y \in \Yse$ there exists an interval $(l,u) \in \wsd$ such that $y$ is the unique optimal solution to \eqref{prob:WS} for any $\lambda \in (l,u)$. Apart from $0$ and $1$ the endpoints of these intervals correspond to the critical weights $\C(\Yse)$. Since the intervals of $\wsd(\Yse)$ are mutually exclusive and collectively exhaustive of $(0,1)$ the total number of critical weights is exactly $|\Yse|-1$ as remarked in \autoref{remark:WsdCellCount}.
  \begin{remark}\label{remark:WsdCellCount} The analysis in \citet{Przybylski2010} immediately implies that $|\Yse| = |\C(\Yse)| + 1$ in the biobjective case. Indeed, $|\C(\Yse)|$ critical weights subdivide the one dimensional weight set $(0,1)$ into $|\C(\Yse)|+1$ weight cells of dimension 1, each of which defines one (unique) extreme supported point by Proposition~4 in \cite{Przybylski2010}.
  \end{remark}

  \begin{remark}[Proposition 3.4 in \citet{paper1}] For any finite subset $\Y\subset \R^2$ and $\lambda>0$, let $\Yl=\arg\max\{\lambda y_1 + (1-\lambda) y_2 : y\in\Y\}$ and define $\Yl\^s$ similarly. Then $|\Yl|=1 \iff |\Yl\^s|=1, \forall s\in \S$.
    \label{remark:YlandYls}
  \end{remark}
  The bound is based on the observation that a weight is critical for $\Yse$ if and only if it is critical for $\Yse\^s$ for some $s\in \S$. \autoref{fig:proof_upper_bound_Yse} shows an example of a weight set decompositions and critical weights for an exemplary set $\Yse$ and for two sub-problem sets $\Yse\^1$ and $\Yse\^2$. Each depicted interval of dimension 1 corresponds to an extreme suported point and one can see how the number of intervals in the weight set decomposition of $\Yse$ can be derived from the weight set decompositions of the sub-problems $\Yse\^s$. We further analyse this in the following proposition.
  
  \begin{prop}\label{prop:boundOnYse}
    Let $\Y\^s \subseteq \mathbb{R}^2$ for $s \in \S$. Then $|\Yse| \le 1 - |\S| + \sum_{s \in \S} |\Yse\^s| $ and the inequality holds with equality if and only if $|\C(\Yse)| = \sum_{s \in \S} |\C(\Yse\^s)|$.
  \end{prop}

  \begin{figure}[h!]
    \begin{center}
      \includegraphics[width=0.95\textwidth]{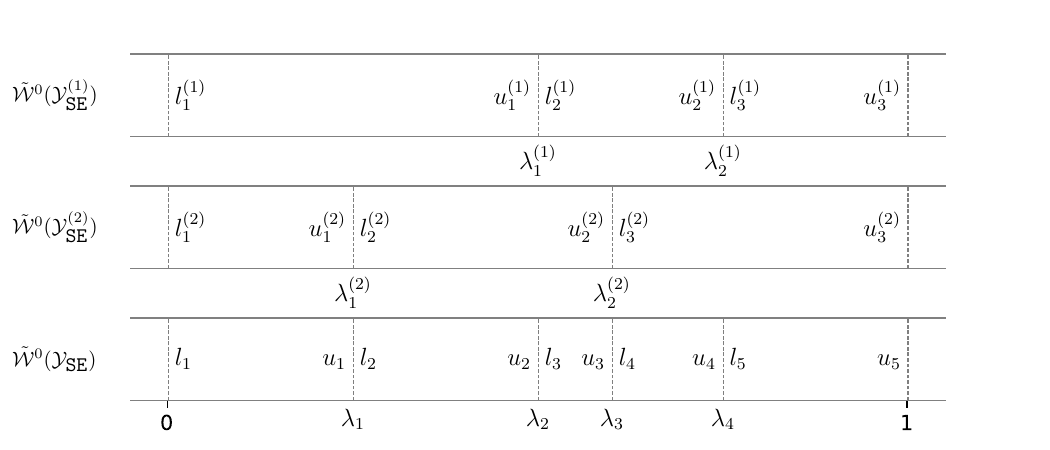}
    \end{center}
    \caption{A visualization of the weight set decompositions $\wsd(\cdot)$ of the sets $\Y\^1, \Y\^2$ and $\Y = \Y\^1 + \Y\^2$, where $|\Yse\^1| = 3, |\Yse\^2| = 3$ and $|\Yse| = 5$.}\label{fig:proof_upper_bound_Yse}  
  \end{figure}

  \begin{proof}
      Using Remark~\ref{remark:YlandYls}, we first show that the set of critical weights $\C(\Yse)$ of the weight set decomposition is equal to the union of the critical weights of the sub-problem weight set decompositions denoted by $\Gamma  \coloneq \bigcup_{s \in \S} \C(\Yse\^s)$.
    Indeed, $\hat{\lambda}\in\C(\Yse)$ if and only if $|\Y_{\hat{\lambda}}|>1$. By Remark~\ref{remark:YlandYls}, the latter is satisfied if and only if $|\Y\^s_{\hat{\lambda}}|>1$ for some $s\in\S$, which is equivalent to $\hat{\lambda}\in \Gamma$. This proves $\C(\Yse)=\Gamma$ and hence $|\C(\Yse)|\leq \sum_{s\in\S}|\C(\Yse\^s)|$. Combining this with Remark~\ref{remark:WsdCellCount} we have
    \begin{align}
        |\Yse| &= 1+|\C(\Yse)|  \\
             &\leq  1+\sum_{s \in \S}|\C(\Yse\^s)|  \\
             &= 1+(\sum_{s \in \S}(|\Yse\^s|-1)  \\
             &= 1 - |S| + \sum_{s \in \S}|\Yse\^s|, 
    \end{align}
    which holds with equality if and only if $|\C(\Yse)|=\sum_{s \in \S}|\C(\Yse\^s)|$.
  \end{proof}

    It is well known that the IP's are NP hard in general.
  IP's being NP-hard implies that multi-objective IP's are NP hard (as they require finding $|\Yn|$ solutions of problems which are NP-hard).
  Likewise the task of finding all extreme supported points is NP-hard as it requires solving at least $|\Yse|$ IP problems.
  The Dichotomic Search algorithm solves $2|\Yse|-1$ IP problems. Using \autoref{prop:boundOnYse} we can assert that the Dichotomic Search algorithm solves at most $2( 1 - |\S| + \sum_{s \in \S} |\Yse\^s|)-1$ IP problems and exactly that many if $|\C(\Yse)| = \sum_{s \in \S} |\C(\Yse\^s)|$. Since solving IP's is computationally expensive and since we are only interested in a subset of the extreme supported solutions we will present a set of algorithms which derive only the subset we are interested in. 

\subsection{Finding the Region of Interest}\label{sec:theory-find-roi}

  In this section we consider the task of determining the points $y^+$ and $y^-$ defining the \roi.
  The points $y^+$ and $y^-$ are defined as follows.
    \begin{align}
    y^+ = \arg\min_{y \in \Yse}\left\{y_2 \mid y_w \ge W\right\}, \quad
    y^- = \arg\max_{y \in \Yse}\left\{y_2 \mid y_w < W\right\} \label{def:yplus_yminus}
  \end{align}

  If $\Yse$ is known — after solving the \phaseone — then $y^+$ and $y^-$ are straight forward to compute by solving \eqref{def:yplus_yminus}.
  However, computing all extreme-supported points $\Yse$ is unnecessary, since one can determine $y^+$ and $y^-$ by calculating only a small subset of $\Yse$. 

  We propose the following modification of the \phaseone which derives $y^+$ and $y^-$ by performing a binary search of the weight set, to determine values $\l^+$ and $\l^-$ producing $y^+$ and $y^-$. The algorithm initiates by determining $y^{ul}$ and $y^{lr}$ (the lexicographical maximal solutions) and iteratively solving scalarized problems moving from the points $y^{ul}$ and $y^{lr}$ towards $y^+$ and $y^-$. The algorithm \findRoi is presented in \autoref{ALG:Finding_ROI}.
  
  In each iteration a solution $y$ is found by solving a weighted sum problem with weight $\lambda$ defined by two incumbant solutions $y^+$ and $y^-$ which are known to be above and below $W$, respectively.
  The weight $\lambda$ is defined such that the solutions corresponding to $y^+$ and $y^-$ have the same objective value in the weighted problem $P_\l$, \ie, $\l y^+_1 + (1-\l)y^+_2 = \l y^-_1 + (1-\l)y^-_2$.
  If instead $y \notin \{y^+,y^-\}$, then the point $y$ lies between $y^+$ and $y^-$. If $y$ is above $W$, \ie, $y_2 \ge W$ we set $y^+ \coloneq y$, otherwise we set $y^- \coloneq y$.
  If $y \in \{y^+,y^-\}$, then no point of $\Yse$ lies between $y^+$ and $y^-$, hence the points $y^+$ and $y^-$ returned by the algorithm solves \eqref{def:yplus_yminus}.

  \begin{algorithm}[h!]
            \KwData{Bi-objective problem $P$ and soft constraint right-hand-side $W$.}
            \KwOut{$y^+$ and $y^-$ satisfying \eqref{def:yplus_yminus}.}
            \BlankLine
            \tcc{Compute lex-max solutions}
            $x^{ul}\in\arg\mbox{lex}\max\{(C_2x,C_1x)\mid x\in\mathcal{X}\}$\;
            $x^{lr}\in\arg\mbox{lex}\max\{(C_1x,C_2x)\mid x\in\mathcal{X}\}$\;
            $y^{ul}\gets Cx^{ul}, y^{lr} \gets Cx^{lr}$\;%
            \If{$y^{ul}=y^{lr}$}{
                \Return $y^{lr}$, $y^{lr}$\tcc*{Only one non-dominated point}
            }
            \BlankLine
            \tcc{Initialize pointers}
            $y^+ \gets y^{ul}$ and $y^- \gets y^{lr}$\;
            $y^*=\texttt{null}$\;
            \While{$y^*\notin \{y^+,y^-\}$}{
                $\lambda \gets \frac{y^+_2-y^-_2}{y^-_1 - y^+_1 - y^-_2 + y^+_2}$\label{ALG:find_ROI_repeat}\;
                $x^*\in\arg\max\{(\lambda C_1+(1-\lambda)C_2)x\mid x\in\mathcal{X}\}$\;
                $y^*\gets Cx^*$\; %
                \uIf{$y_2^*\geq W$}{
                    $y^+\gets y^*$
                }
                \Else{
                    $y^-\gets y^*$
                }
            }
            \Return $y^+$, $y^-$
            \caption{ \findRoi }\label{ALG:Finding_ROI}
  \end{algorithm}

  \begin{prop}\label{prop:find-roi-runtime} Let $T(n,m,M)$ be the time complexity of an IP problem with $n$ variables, $m$ constraints and \(M\) being the largest entry of $A,b$ and $c$. If all cells in the weight set decomposition are the same size, then the worst-case running time of \findRoi is $(3+\log_{2}(|\Yse|))\cdot T(n,m,M)$.
  \end{prop}
  \begin{proof} %
  The algorithm always solves two IP's, one for each lex-max solution. Then, assuming all cells are the same size, the algorithm performs a binary search on the cells of the weight set. Since the cells are assumed to be all the same size, exactly half of the points are excluded in each iteration. Therefore, at most $\log_2(|\Yse|)$ iterations of the main loop are performed before $y^+$ and $y^-$ are correctly identified. A final IP call is made showing that there are no supported points between $y^+$ and $y^-$, resulting in a total of at most $(2+\log_2(|\Yse|)+1)$ IP calls.
   \end{proof}

 In \autoref{prop:find-roi-runtime} we showed a bound on the running time of \findRoi given the assumption that the cells of the weight set decomposition are all the same size. We are not suggesting that this assumption holds in general, but the theoretical bound is a good predictor of the expected number of IP calls, as we will see in the empirical study (see \autoref{fig:emp_IP_calls_Yse}). In general, the worst case performance of \findRoi coincides with that of the \phaseone, which requires $2|\Yse|-1$ IP calls.

    In the setting where the problem $P$ is additively-separable into sub-problems $\left\{ P\^s \right\}_{s=1}^S$ one can decompose the calculation of weighed sum solutions in \findRoi by applying the following lemma:

    \begin{lemma}[\citet{paper1}] Let $(\X,\Y, f)$ be decomposable into $(\X\^s, \Y\^s, f\^s)$ for $s \in \S$ such that $\X = \prod_{s\in \S} \X\^s$ and $\Y = \msum_{s\in \S}\Y\^s$. Then $P_\l = \msum_{s \in \S} P\^s_\l$
    \end{lemma}

  In particular one can decompose the solution process of $y^{ul}$ and $y^{lr}$ by solving the corresponding lex-max solutions $y^{s,ul}$ and $y^{s,lr}$ for each sub-problem $s \in \S$ and setting 

  \begin{align}
    y^{ul} = \sum_{s\in \S} y^{s,ul},\quad \text{ and } \quad y^{lr} = \sum_{s\in \S} y^{s,lr} \label{eq:decomposed_ul_lr}
  \end{align}

  In fact, for any $\lambda \in(0,1)$ a solution $x^\lambda$ can be derived by finding $y^{s,\lambda} \in P\^s_\lambda$ for each $s \in \S$ and thereafter setting

  \begin{align}
    x\^\lambda &= (x\^1,\ldots, x\^S) = (x\^1_1,x\^1_2, \ldots, x\^1_{n\^1},\ldots, x\^S,\ldots x\^S_{n\^S}) \label{eq:decomposed_x_lambda}
  \\
    y\^{\lambda} &= \sum_{s\in \S} y\^{s,\lambda} \label{eq:decomposed_y_lambda}
  \end{align}
  To ease the notation of the paper, we will suppress writing $x$ and instead we will say that an objective vector $y$ is optimal to some problem, if there exists an optimal feasible solution $x$ to the problem for which $Cx = y$.
  
  Since general IP's are known to be NP-hard problems, it is advantageous to decompose them into a set of smaller sub-problems. 
  It is clear that the computational complexity of solving a set of $|\S|$ smaller sub-problems — where each sub-problem $s \in \S$ has $n\^s$ variables and $m\^s$ constraints — is easier than solving a single IP with $n = \sum_{s\in \S}n\^s$ variables and $m = \sum_{s\in \S}m\^s$ constraints.

  Hence, the idea of the decomposition algorithm is to run a version of \findRoi which determines solutions $y^\l$ by solving problems $WS(\lambda)\^s$ for each $s\in \S$, instead of solving the larger \eqref{prob:WS} problem.

  Throughout the iterations of the algorithm it might happen that a sub-problem is solved for different values of $\lambda$ which provide the same solution. \autoref{lemma:lambda-lookup} states that if two $\lambda$-values, $\lambda_1$ and $\lambda_2$, have the same solution, then any $\lambda$-value in the interval between them will have the same solution.

  \begin{lemma}\label{lemma:lambda-lookup}
    Let $x \in \X$ and assume $x \in \aWS{\lambda_1}$ and $x \in \aWS{\lambda_2}$ for $\lambda_1,\lambda_2 \in (0,1)$ with $\lambda_1 < \lambda_2$. Then $x \in \aWS{\lambda}$ for any $\lambda \in [\lambda_1, \lambda_2]$.
    \label{lemma:lambda_lookup}
  \end{lemma}
  \begin{proof}
    \newcommand{\lambdavecc}{(\lambda, 1- \lambda)}
    Assume for contradiction $\exists \bar{\lambda}\in (\lambda_1,\lambda_2)$ where $x \notin \aWS{\bar{\lambda}}$. Let $\bar{x}\in \aWS{\bar{\lambda}}$ and consider the linear function $h(\lambda) = \lambdavecc Cx - \lambdavecc C\bar{x}$.
    Then $h(\lambda_1)\ge 0$, $h(\bar{\lambda}) < 0$ and $h(\lambda_2)\ge 0$, which would contradict $h$ being a linear function since $\lambda_1<\bar{\lambda} <\lambda_2$. 
  \end{proof}

  The algorithm \findRoiDecLook is presented in Algorithm~\ref{ALG:Finding_ROI_lambda}. In this decomposed version of \findRoi, solutions $y\^\l$ are computed as in \eqref{eq:decomposed_y_lambda}.
  The algorithm makes use of \autoref{lemma:lambda_lookup} as follows:
  Any time a sub-problem $\WSs{\lambda}{s}$ is solved, the $\lambda$-value and its corresponding solution is saved in a set $\Lambda\^s$. 
  When a new problem $\WSs{\lambda}{s}$ is to be solved for some $s\in \S$, the algorithm checks in the set $\Lambda\^s$ if there exists two weights $\l_1$ and $\l_2$ both mapping to the same solution $\bar{y}\^s$ such that $\l \in [\l_1, \l_2]$. We refer to this check as \LookupLambda.  The subroutine checks if such a pair of weights is stored in a set $\Lambda\^s$ and returns the corresponding solution if one exists. If \LookupLambda returns a solution $\bar{y}\^s$ then the algorithm makes use of \autoref{lemma:lambda_lookup} to skip calls to the IP-solver and instead loads $\bar{y}\^s$ as the optimal solution to $\WSs{\l}{s}$. We call the process of checking previous calls and loading solutions $\lambda$-lookup.
  If a solution is loaded this way, the algorithm skips solving an NP-hard IP problem, at the cost of the lookup time.

      \begin{algorithm}[h!]
            \KwData{Bi-objective problem $P$ and soft constraint right-hand-side $W$.}
            \KwOut{$y^+$ and $y^-$ satisfying \eqref{def:yplus_yminus}.}
            \BlankLine
            \For{$s\in\S$}{
                \tcc{Compute lex-max solutions for each sub-problem}
                $x^{s,ul}\in\arg\mbox{lex}\max\{(C_2\^sx\^s,C_1\^sx\^s)\mid x\^{s}\in\mathcal{X}\^s\}$\;
                $x^{s,lr}\in\arg\mbox{lex}\max\{(C_1\^sx\^s,C_2\^sx\^s)\mid x\^{s}\in\mathcal{X}\^s\}$\;
                $y^{s,ul}\gets C\^sx\^{s,ul}, y^{s,lr}\gets C\^sx\^{s,lr}$\;%
                $\Lambda\^s \gets \emptyset$
            }
            \BlankLine 
            \tcc{Initialize pointers}
            $y^{ul} \gets \sum_{s \in \S } y^{s,ul}$ and $y^{lr} \gets \sum_{s\in \S} y^{s,lr}$\;
            $y^+ \gets y^{ul}$ and $y^- \gets y^{lr}$\;
            $y^*=\texttt{null}$\;
            \While{$y^* \notin \{y^+,y^-\}$}{
                $\lambda \gets \frac{y^+_2-y^-_2}{y^-_1 - y^+_1 - y^-_2 + y^+_2}$\label{ALG:find_ROI_lambda_repeat}\; 
                \For{$s\in\S$}{
                  \uIf{\LookupLambda returns a solution $\bar{y}\^s$}{
                  ${y\^s}^* \gets \bar{y}\^s$\tcc*{We already know ${\bar{y}}\^s$ is optimal}
                  }
                  \Else{
                    \tcc{Need to compute optimal solution}
                    ${x\^s}^*\in\arg\max\{(\lambda C_1\^s+(1-\lambda)C_2\^s)x\^s\mid x\^{s}\in\mathcal{X}\^s\}$\;
                    ${y\^s}^*\gets C\^s{x\^s}^*$\;%
                    $\Lambda\^s \gets \Lambda\^s \cup \{(\lambda,{y\^s}^*)\}$\tcc*{Update cell information}
                  }
                }
                $y^* \gets \sum_{s\in \S}{y\^s}^*$ \;
                \uIf{$y_2^* \ge W$}{
                  $y^+ \gets y^*$
                }
                \Else{
                  $y^- \gets y^*$
                }
            }
            \Return{$y^+, y^-$}
            \caption{ \findRoiDecLook}\label{ALG:Finding_ROI_lambda}
    \end{algorithm}

    In \autoref{prop:find-roi-dec-runtime} we show how the running time of \findRoiDecLook depends solely on the size of the sub-problems.
    \begin{prop}\label{prop:find-roi-dec-runtime} Let $T(n,m,M)$ be the time complexity of a \bsip problem with $n$ variables, $m$ constraints and \(M\) being the largest entry of $A,b$ and $c$. If all cells of the weight set decomposition $\wsd(|\Yse|)$ are the same size, then the worst-case running time of \findRoiDecLook is $(3+\log(|\Yse|))\cdot \sum_{s\in \S}T(n\^s,m\^s,M\^s)$.
    \end{prop}
    \begin{proof} Follows from \autoref{prop:find-roi-runtime} and \eqref{eq:decomposed_x_lambda}.
    \end{proof}

    In \autoref{sec:comp-study} we present an empirical study evaluating \findRoiDecLook. There we study the effect of solving weighted sum problems in the decomposed way, as well as the effect of the $\lambda$-lookup idea.

  \subsubsection{Approximation quality for the original single objective problem}\label{sec:bound-measures}

  If the decision maker is only interested in an approximation of an optimal solution to the original problem, then $y^+$ and $y^-$ would constitute such approximations for the optimal solution $y^*$.
  The point $y^+$ would be suboptimal but feasible, while
  $y^-$ would be better than the optimal solution $y^*$ but infeasible (superoptimal).

  In general $y^* \in \Rectangle\left[y^{+}, y^{-}\right]$, hence knowing $y^+$ and $y^-$ will give bounds on the error.
  We define the point $y^W$ as the intersection between the lines $y_2 = W$ and the line connecting $y^+$ and $y^-$. Note that in general $y^W$ is not feasible. This point satisfies $y^W_2 = W$ and $y^W = \lambda y^+ +(1-\lambda) y^-$ for %
  $\lambda = (W-y^-_2)/(y^+_2-y^-_2)$
  , see \autoref{fig:error_y_W}.
  We know that $y^*$ lies in the triangle defined by $y^+$ and $y^W$, visualised 
  in \autoref{fig:error_y_W}.
  Using this, we will define the bound on the error ($\verb|error bound|$), as well as the $\verb|actual error|$ (distance to $y^*$).
  \begin{align}
    e_1^{actual} \coloneq |y^*_1 - y^+_1| \le |y^W_1 - y^+_1| \eqcolon e_1^{bound}  &\text{ (The optimality error)} \\
    e_2^{actual}\coloneq |y^*_2 - y^+_2| \le |y^W_2 - y^+_2| \eqcolon e_2^{bound}  &\text{ (The feasibility surplus)} 
  \end{align}
  We normalise the errors by using the lex-max solutions and define normalization weights $e^{max}_1$ and $e^{max}_2$ %
  as $e^{max}_1 := |y^{ul}_1 - y^{lr}_1|$ and $e^{max}_2 := |y^{ul}_2 - y^{lr}_2|$.
  We add a bar over an error to indicate that it is normalised, \ie, $\bar{e}_{p}^{type} \coloneq e_{p}^{type}/e^{max}_p$ for $p\in \{1,2\}$ and $type\in \{actual, bound\}$.

    \begin{figure}[tbp]
        \centering
        \begin{tikzpicture}
            \usetikzlibrary{decorations.pathreplacing}
            \node[draw,circle,minimum size=0.25cm,inner sep=0pt, fill=black] (a) at (0,6) {};
            \node[draw,circle,minimum size=0.25cm,inner sep=0pt, fill=black] (b) at (2,5.3) {};
            \node[draw,circle,minimum size=0.25cm,inner sep=0pt, fill=black] (c) at (5,3.5) {};
            \node[draw,circle,minimum size=0.25cm,inner sep=0pt, fill=black] (d) at (6,2) {};
            \node[draw,circle,minimum size=0.25cm,inner sep=0pt] (opt) at (2.8,4.5) {};
            \node[draw,circle,minimum size=0.25cm,inner sep=0pt, fill=black] (yw) at (4.2,4) {};
            \node[draw,circle,minimum size=0.25cm,inner sep=0pt, fill=black] (ybest) at (2.2,4.75) {};
            
            \node[] (optY) at (4,6) {$y^*$};
            \draw[->,shorten >=3pt] (optY) to[out=270, in=0] (opt);
            
            \node[] (ywName) at (5,5) {$y^W$};
            \draw[->,shorten >=3pt] (ywName) to[out=200, in=90] (yw);

            \node[] (ybestName) at (3,6.5) {$y^{best}$};
            \draw[->,shorten >=3pt] (ybestName) to[out=270, in=45] (ybest);
            
            \coordinate (corner) at (2,3.5);
            \draw[dashed] (b)--(c)--(corner)--(b);

            \draw[dotted,thick] (-0.5,4)--(6.5,4);
            \node[above] at (-0.5,4) {$W$};
            \node[above] at (a) {$y^{ul}$};
            \node[above] at (b) {$y^{+}$};
            \node[right] at (c) {$y^{-}$};
            \node[right] at (d) {$y^{lr}$};

            \draw[decorate, decoration={brace}] (1.8,4.01) -- (1.8,5.3) node[midway, left] {$e^{bound}_2$};
            \draw[decorate, decoration={brace}] (4.2,3.3)--(2,3.3) node[midway, below] {$e^{bound}_1$};
        \end{tikzpicture}
        \caption{Illustration of the bound measures. The optimality error $e^{bound}_2$ measures how far $y^+$ is from an optimal solution value. The feasibility surplus measures the slack in the soft constraint at the solution corresponding to $y^+$. The point $y^*$ illustrates the optimal solution, while $y^{best}$ shows the best solution found (not necessarily optimal).}
        \label{fig:error_y_W}
    \end{figure}
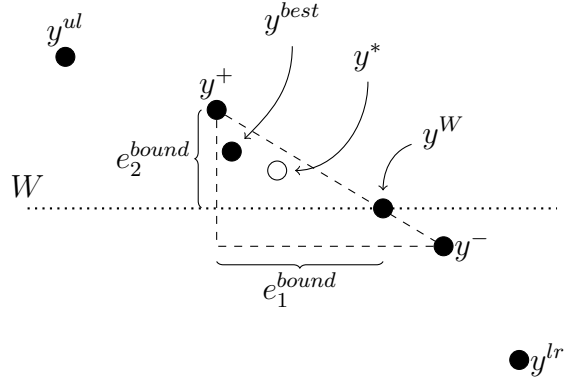

  A decision maker might consider the error bound sufficiently small, and choose a solution generated in the process of finding the defining points $y^+$ and $y^-$. Alternatively, if the error gap is too large, or if the decision-maker is interested in generating all points in the \roi, then more computation is needed. In the following subsection, we present an algorithm which finds all points of the \roi.

\subsection{Finding solutions in the region of interest}\label{sec:theory-solve-roi}

  In this subsection, we present an algorithm for finding all \nd-points in the ROI $R$, denoted $\Yr = \Yn \cap R$.
  The idea of the algorithm \solveRoi presented in \autoref{ALG:Solving_ROI} is to solve a sequence of augmented $\varepsilon$ constraint problems of the form
      \begin{equation}\tag{$P\^{s,d,\epsilon}$}
        \max_{x\in \X}\left\{ C\^s_{\bar{d}}x + \frac{1}{\alpha} C\^s_{d}x \mid A\^sx\le b\^s, C\^s_dx\^s \ge \epsilon \right\},
    \label{eq:epsilonLex}
  \end{equation}
  thereby iteratively finding \nd solutions of \SP{}s, until all \SP{}s are solved or a termination criteria is reached.

  Each time a new \SP solution is found the representation $\hat{\mathcal{Y}}_R$ of $\mathcal{Y}_R$ is updated. If all sub-problem \nd solutions are known then $\hat{\mathcal{Y}}_R=\mathcal{Y}_R$. If however, a stopping condition is reached before the \nd solutions of \SP{}s are found, then the quality of the representation $\hat{\mathcal{Y}}_R$ depends on the sequence in which \SP solutions are found.
  \eeg if only the \nd points around the lex-max \SP solutions are known, then the representation $\hat{\mathcal{Y}}_R$ likewise can be expected to consists of points around the lex-max solutions only.
  Before running the \solveRoi algorithm, we assume that the ROI-defining points $y^+$ and $y^-$ are known. Furthermore, we assume to know the \SP points generating these, \ie points $y^{+,s}$ and $y^{-,s}$ for each $s \in \S$ such that $y^+ = \sum_{s \in \S} y^{+,s}$ and $y^- = \sum_{s \in \S} y^{-,s}$.
  The \solveRoi algorithm solves a sequence of sub-problems starting from the \emph{centre} of \SP{}s and moving \emph{outwards} in a bi-directional way, with the goal that the representation $\hYr$ \emph{converges} faster towards $\Yr$ when comparing against a uni-directional sequence. 
  The heuristic idea is that if the known \SP solutions are concentrated around $y^{+,s}$ and $y^{-,s}$ for each $s \in \S$, then we can expect the representation $\hat{\mathcal{Y}}_R$ to be concentrated around $y^+$ and $y^-$ as well, 
  resulting in a better representation of $\Yr$.
  \autoref{fig:AltEps} shows a visualization of two different sequences for finding \ND solutions to a bi-objective problem: a (classical) uni-directional approach and the proposed bi-directional approach.

  New sub-problem solutions are found by solving the augmented $\varepsilon$-constraint lex-max problems defined in \eqref{eq:epsilonLex}, which finds the \emph{next} \nd-point in a given direction~$d$ such that the value of the $d$th objective is greater than $\varepsilon$, where $\varepsilon$ is initially defined by $y\^{s,+}$ and $y\^{s,-}$ for each $s\in \S$.
  This resembles the ideas presented by \citet{chalmet1986algorithm} and further developed by \citet{boland2015}.

  In the following computational study, we will set a maximum number of IP-calls as a stopping criterion. Other possible stopping criteria include a maximum time limit or a minimum error bound, based on the error measures presented in the previous section.

  \begin{figure}
      \begin{center}
\begin{subfigure}[b]{0.45\textwidth}
          \includegraphics[width=1\textwidth]{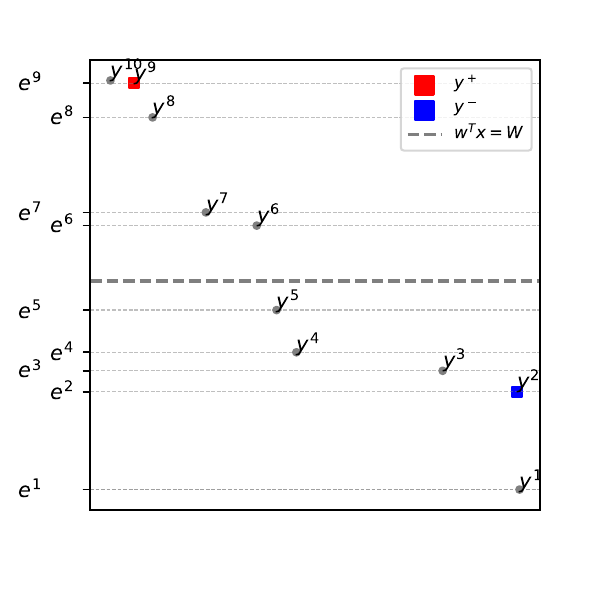}
          \caption{$\varepsilon$-constraint method sequence. $y^i$ is the $i$th point found in the sequence.}
\end{subfigure}
\begin{subfigure}[b]{0.45\textwidth}
          \includegraphics[width=1\textwidth]{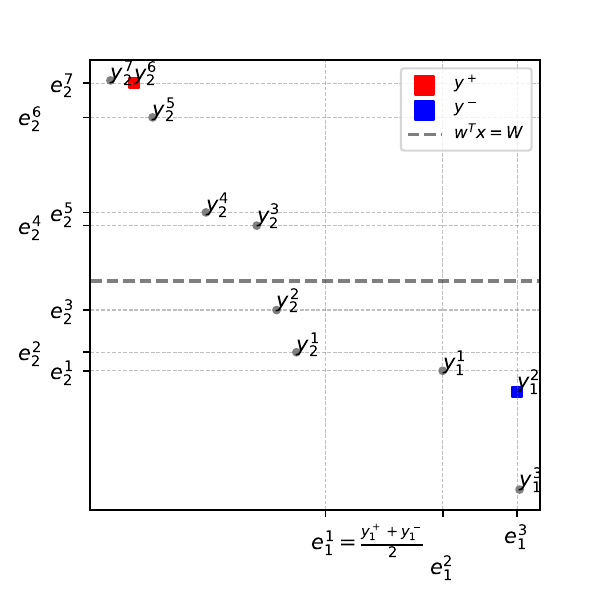}
            \caption{bi-directional $\varepsilon$-constraint method sequence. $y^i_d$ is the $i$th point found in direction $d \in \{1,2\}$.}
\end{subfigure}
    \caption{A visualization comparing the sequence obtained by the uni-directional $\varepsilon$-constraint method and the bi-directional $\varepsilon$-constraint method, which generates points sequentially, starting from the centre of the rectangle $\Rectangle[y^+, y^-]$ and moves outwards.}\label{fig:AltEps}
        \end{center}
  \end{figure}

  \begin{algorithm}[h!]
            \KwData{Bi-objective problem $P$,
            \ROI defining points $y^+,y^-$ along with $y^{+,s}$ and $y^{-,s}$ for each $s \in \S$ such that $y^+ = \sum_{s \in \S} y^{+,s}$ and $y^- = \sum_{s \in \S} y^{-,s}$. $\hat{\Yn}\^s \supseteq \{y\^{s,+}, y\^{s,-}\}$
            }
            \KwOut{A representation $\hYr$ of $\Yr$ if an early stopping condition reached, otherwise $\hYr=\Yr$.}
            $unsolved \gets \S \times \{1,2\}$ \tcc{sub-problem and direction pairs}
            \While{$unsolved \neq \emptyset$}{
              choose $(s,d) \in unsolved$ \tcc{choose unsolved sub-problem and direction, using some selection strategy}
              let $\bar{d} \in \{1,2\}\setminus \{d\}$\label{alg-line:node-selection}\;
              \If{this is the first time sub-problem $s$ is chosen.
              }{
                $\epsilon \gets \frac{1}{2}\left(y^+_{\bar{d}} + y^-_{\bar{d}}\right)$ \tcc{start from sub-problem centre}
              }
              \ElseIf{$y\^{s,d}_{\bar{d}}$ is defined}{
                $\epsilon \gets y\^{s,d}_{\bar{d}} + 1$
                \tcc{define $\epsilon$ from previous solution in direction $d$}
              }
              \Else{
                $\epsilon \gets y\^{s,\bar{d}}_{d} + 1$ \tcc{if first time searching direction $d$ for sub-problem $s$}
              }
              choose $\alpha>0$ sufficiently large\;
                \If{\eqref{eq:epsilonLex} is feasible\label{ALG:Solving_ROI-problem}}
              {
                $x\^{s,d}\gets $ an optimal solution to \eqref{eq:epsilonLex}\;
                $y\^{s,d} \gets C\^sx\^{s,d}$\;
              }
              \Else{
              $unsolved \gets unsolved \setminus \{(s,d)\}$
              }
              $\hYr \gets \nondom{\msum_{s\in\S}\hat{\Yn}\^s}$ \tcc{update representation}
            \If{early stopping condition is reached}{
              \Return{$\hYr$}
            }
            }
            \Return{$\Yr$}
            \caption{ \solveRoi }\label{ALG:Solving_ROI}
  \end{algorithm}

\section{Computational study}\label{sec:comp-study}
\newcommand{\nr}[1][x]{{\color{black}#1}\@\xspace} %
\newcommand{\nrp}[1][x]{{\color{black}#1\%}\@\xspace} %

  In this computational study, we wish to investigate the effectiveness of the above proposed algorithms. We want to answer the following empirical research questions:
  \begin{enumerate}
    \item Finding the \ROI.
      \begin{enumerate}
        \item[1.] How do the proposed methods for finding the region of interest perform with respect to time and IP calls compared to each other and to the \phaseone?
        \item[2.] How well does the provided solutions approximate the optimal solution? %
    \end{enumerate}
  \item Solving the \ROI.
    \begin{enumerate}
      \item[1.] Does a bi-directional method outperform an uni-directional method when solving the \roi-problem?%
      \item[2.] What is the effect of different sub-problem selection methods on the performance of \solveRoi?%
    \end{enumerate}
  \end{enumerate}
  To do this, we create a test-bed of \bsip instances based on sub-problem instances from the literature.
  The instances and details of the implementation are discussed in \autoref{sec:test-instances}. In \autoref{sec:comp-find}, we look at Research Questions~1.1 and 1.2. %
  Then, in \autoref{sec:comp-solve}, we investigate the performance of \solveRoi  by answering Research Questions~2.1 and 2.2. %

\subsection{Test instances and implementation}\label{sec:test-instances}

  To test the algorithms, we create a set of \nr[384] block-structured IP instances available at \citet{BSIPInstances}. 
  The test instances are created by combining blocks of smaller IP problems — specifically multi-dimensional knapsack problems (MKP) and assignment problems (AP). The MKP instances are taken from \citet{drake2015benchmark} and the AP instances are taken from \citet{Forget20}. The set of these sub-problems form a test-bed, where each instance has been classified into type (MKP/AP) and difficulty level (1--easy and 2--hard, defined by the solution time). 

  Given a set of sub-problems $P\^s = (A\^s, b\^s, c\^s)$ for $s \in \S$, one can define a \bsip $(A,b,c)$ by stacking the corresponding vectors/matrices of the sub-problems. Specifically, $A$ is defined by stacking all matrices $A\^s$ diagonally, while $b$ and $c$ are defined by stacking the corresponding vectors vertically and horizontally, respectively. See \autoref{table:visual-instances} for a visualization of how problems are combined.

  We define the coefficients of the soft constraints following an approach similar to that of \citet{schulze2017soft}. 
  Given some objective space coefficient vector $c \in \mathbb{N}^n$ we will define a constraint coefficient vector $w \in \mathbb{N}^n$ in the following two ways. In the random approach, we generate the coefficients as $w_i \sim \unif(1, 1/n\sum_{i}^nc_i)$. In the deterministic approach, we define $w$ as the vector consisting of the inversely ordered entries of $c$. 
  
  Lastly, the right-hand-side value of the soft constraint $wx \ge W$ is calculated as $W := \frac{y^{ul}_2 - y^{lr}_2}{2}$. This is the ``midpoint'' between the two lex-max solutions $y^{lr}, y^{ul}$ of the multi-objective problem $\max\{(cx, wx) | Ax\le b\}$. This requires solving the lex-max solutions for each problem, but is necessary to ensure that the region of interest is well-defined.

  To get a diverse test-bed, we run the instance generation process with eight different seeds for the pseudo random number generator. An overview of all instance configurations is given in Table~\ref{tab:test-instances}. The test-bed of all instances are available at \citet{BSIPInstances}.

  \begin{table}[t]
          
  \centering
  \caption{Overview of instance configurations.}
  \label{tab:test-instances}
  \begin{tabularx}{\linewidth}{lX}
  \toprule
  Configuration of instances& Values / Description \\
  \midrule
  Sub-problem types &  Assigment problems (AP) and Multi-dimensional Knapsack problems (MKP) \\

  Region of interest sizes & $\gamma \in \{0,\; 0.1,\; 0.2\}$ \\

  Seed(s) & $\{1,2,3,4,5,6,7,8\}$\\[2pt]

  \# sub-problems & $\{2,\; 4,\; 6,\; 8\}$ \\

  Difficulty level%
  & $\{1,\; 2\}$ \\

  $w$ coefficients & \{Random, Deterministic inverse\} \\
  \bottomrule
  \end{tabularx}
  \end{table}

  All algorithms were implemented in Python 3.13. 
  The IP problems were solved using CPLEX 22.1.0. CPLEX was run in a single thread, otherwise all default settings were used. %
  For an efficient calculation of the Minkowski sums of non-dominated points, the C-implementation from \citet{klamroth2024} was used.
  Experiments were run on a macOS machine with an M2-processor and 16 GB ram.
  All implementations can be found on the GitHub repository \citet{bsip-Lyngesen26}. %

\subsection{Performance of the find-ROI algorithm}\label{sec:comp-find}

    \begin{figure}
      \centering
          \includegraphics[width=\linewidth]{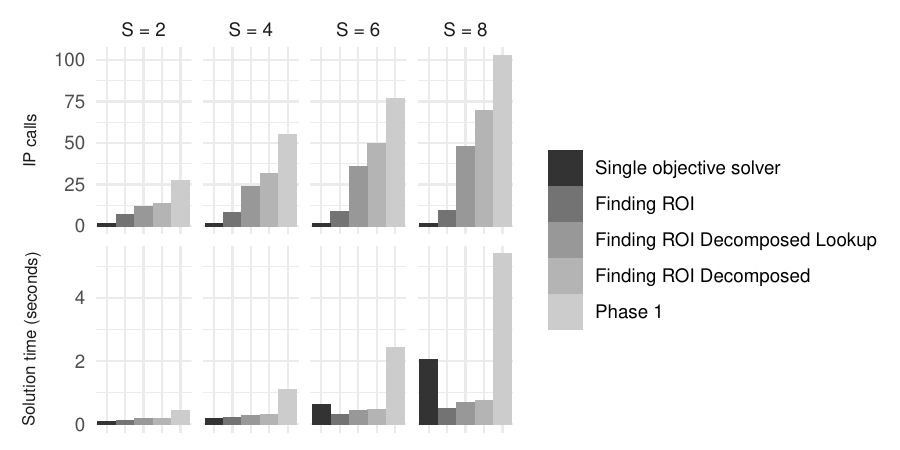}
  \caption{Average IP-calls (first row) and running times (second row) over all \bsip instances, with columns for different number of sub-problems (S).}\label{fig:empFind}
    \end{figure}
    
    In this subsection, we answer the research questions 1.1 and 1.2 related to the task of finding the \roi.
    We ran the algorithms \phaseone, \findRoi and \findRoiDecLook for the set of test instances described in \autoref{sec:test-instances}. 

\subsubsection{Comparing IP calls and running times}
    
    To evaluate the effect of the $\lambda-$lookup of \findRoiDecLook, we also report results for a version of the decomposition \findRoiDec, which does not skip solved intervals of sub-problems. In \autoref{fig:empFind} we show a comparison of the average number of IP calls (first row) along with the average running time (second row) for each test over all instances. Average IP calls and solution time are reported in \autoref{table:emp1-approximation}.

    When comparing the three proposed methods (\autoref{fig:empFind}), it is seen that the binary search without decomposition \findRoi solves fewer IP problems as the ones with decomposition but that each of these sub-problems is computationally more difficult. Looking at the average running time performance (\autoref{table:emp1-approximation}), we observe that \findRoi is the fastest of the tested algorithms, especially for instances with many sub-problems as can also be seen in \autoref{fig:empFind}. Hence, for the tested instances decomposing the calculation of weighted sum solutions did not prove to be faster. This could be a result of an overhead associated with making each IP call: making $S$ small IP-calls seems to be slower than a single large IP-call.
    Comparing the decomposition algorithms \findRoiDec and \findRoiDecLook, we see that there is an efficiency gain when using the $\lambda$-lookup method, where an average of \nrp[24.1] sub-problem IP-calls are skipped when using the $\lambda$-lookup. Additionally, the proportion of IP-calls skipped seems to increase in the number of sub-problems, being  \nrp[12.94], \nrp[23.98], \nrp[28.38] and \nrp[31.17] for $2,4,6$ and $8$ sub-problems, respectively (see \autoref{fig:empFind}).

  \begin{figure}
      \centering        
          \includegraphics[width=1\textwidth]{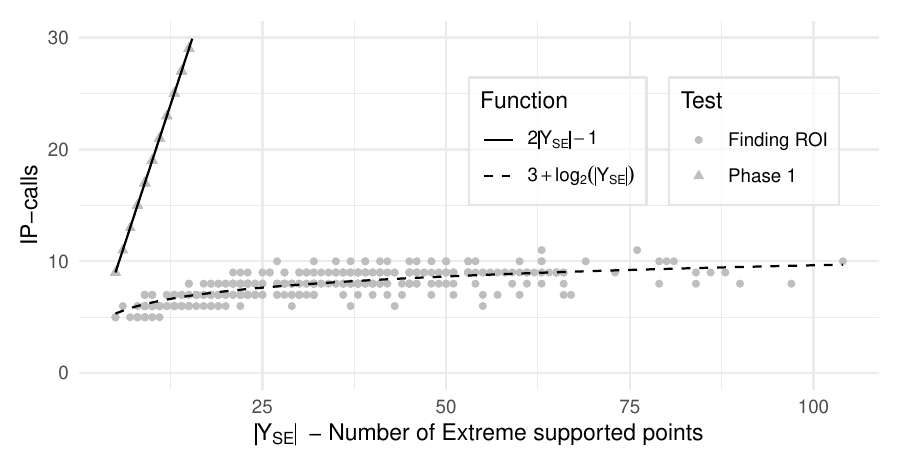}
        \caption{Number of IP calls as a function of $|\Yse|$.}
   \label{fig:emp_IP_calls_Yse}
    \end{figure}
    We notice that the proposed algorithms for finding the \roi defining points $y^+$ and $y^-$ are faster and require solving significantly fewer IP problems compared to running the \phaseone.
    In \autoref{prop:find-roi-runtime} we showed that \findRoi requires solving at most $3+\log2(|\Yse|)$ IP's under the assumption that the cells of the weight space decomposition are all the same size. In \autoref{fig:emp_IP_calls_Yse} we show the number of IP calls as a function of $|\Yse|$ based on the test instances. Here we find that the theoretical worst case (under the strong assumption), serves as a decent approximation of the average number of IP calls.

\subsubsection{Approximation of optimal solutions}

    We find that on average, solving the original single objective \bsip is slower than finding the \roi defining points. For a large number of sub-problems~$S$, the time needed for solving the single objective problem is much higher than for the methods finding the \roi, as can be seen in \autoref{fig:empFind}. This is surprising, as solving the original \bsip does not provide any information of the trade-offs between the objective function and the soft coupling constraint. The \findRoi does not, however, find an optimal solution to the (original) constrained problem. It is therefore of interest to investigate the quality of the best feasible solutions found by \findRoi. 

    We focus on how well the non-dominated solutions found by the algorithms serve as approximations of the single objective optimal solution $y^*$. For each problem the best solution $y\^{best}$ is chosen as either $y^+$ or the best point $y$ found with respect to maximizing $cx$ while satisfying $w\ge W$. In \autoref{table:emp1-approximation} we report the average optimality errors $\bar{e_1}$ as well as the soft constraint surplus $\bar{e_2}$ for each method averaged over all test instances. These error measures are discussed in \autoref{sec:theory-find-roi}.

    We find that the average optimality error $\bar{e}_1^{actual}$ of \nrp[1.54] and the optimality bound $\bar{e}_1^{bound}$ of \nrp[2.29] are relatively small for each method. Likewise, we find that the feasibility error $\bar{e}_2^{actual}$ of \nrp[4.83] and the feasibility error bound $\bar{e}_2^{bound}$ of \nrp[5.97] are small.

    If the computation of weighted sum problems is decomposed into solving sub-problems, a set of sub-problem solutions is known. It then can happen that such sub-problem solutions combine into a global solution which is closer to $y^*$. Because of this, knowing more sub-problem solutions results in lower optimality errors. This can be seen in \autoref{table:emp1-approximation} where the decomposed binary search of \findRoiDecLook results in an optimality gap of around \nrp[0.51], while knowing all extreme supported points of sub-problems results in the best optimality error of \nrp[0.74]. 
    Additionally, we see that in \nrp[32.03] of the solved instances the optimal solution was found when solving in the decomposed way, while this was true for only \nrp[15.89] instances when the problem was not decomposed.
    
Concluding, we can say that the methods \findRoi and \findRoiDecLook both can serve as meaningful heuristics for finding approximate solutions to \bsip's. It is interesting to note here that \findRoi would return the feasible solution $y^+$ which corresponds to the solution obtained from a Lagrangian relaxation of the problem. Additionally, \findRoi would return a super-optimal solution $y^-$, which together with $y^+$ would provide a provable optimality gap along with trade-off information. The other proposed method \findRoiDecLook would, at a slightly higher computational cost, provide the same information, and additionally by combining the solutions of sub-problem a potentially better solution $y\^{best}$ is obtained.
    \begin{table}
      \footnotesize
    \caption{
    The table presents error measures for each method averaged over all test instances.
    Numbers are written in \% except the columns with average solution time in seconds (Time) and the average number of IP calls (IP). 
    The column `$y^*$ found` shows the proportion of instances where the optimal solution to \bsip was found. 
    The column opt\_gap shows the optimality gap $\frac{y^w_1-y^{best}_1}{y^w_1}$ (in the traditional sense), where $y^w_1$ is an upper bound and $y_1^{best}$ is a lower bound for $y^*_1$. }\label{table:emp1-approximation}
      \centering
      \begin{tabularx}{\linewidth}{p{1.5cm}rrrrrrrr}\toprule
      Method & $\bar{e}_1^{actual}$ & $\bar{e}_1^{bound}$ & $\bar{e}_2^{actual}$ & $\bar{e}_2^{bound}$& opt\_gap &$y^*$ found & IP & Time\\
      \midrule

Single objective solver & NA & NA & NA & NA & NA & NA & 1.00 & 0.73\\
\midrule
\findRoi & 1.54 & 2.29 & 4.83 & 5.97 & 0.74 & 15.89 & 7.78 & 0.28\\
\midrule
\findRoiDecLook & 1.54 & 2.29 & 4.83 & 5.97 & 0.51 & 32.03 & 29.49 & 0.38\\
\midrule
\findRoiDec & 1.54 & 2.29 & 4.83 & 5.97 & 0.51 & 32.03 & 40.67 & 0.41\\
\midrule
\phaseone & 1.54 & 2.29 & 4.83 & 5.97 & 0.74 & 15.89 & 65.19 & 2.32\\
      \bottomrule
      \end{tabularx}
    \end{table}

\subsection{Performance of the solve-ROI algorithm}\label{sec:comp-solve}

  Here, we investigate the proposed algorithm for finding all points in the region of interest answering research questions 2.1 and 2.2. Before computing all points in the \roi, \findRoiDecLook needs to find the defining points $y^+$ and $y^-$ for the region of interest. This means that before running \solveRoi a subset of $\Yr$ is already known. In the computational experiments, three different sizes $(\gamma)$ of \roi are tested as shown in $\autoref{tab:test-instances}$. To reduce the time needed to run the experiments only half of the seeds were solved, and instances with $8$ sub-problems were excluded, resulting in a total of \nr[432] instances.

  \subsubsection{Bi-directional vs uni-directional search}\label{sec:node-selection}

    \begin{figure}
      \centering        
      \begin{subfigure}[t]{0.47\linewidth}
          \includegraphics[width=0.95\textwidth]{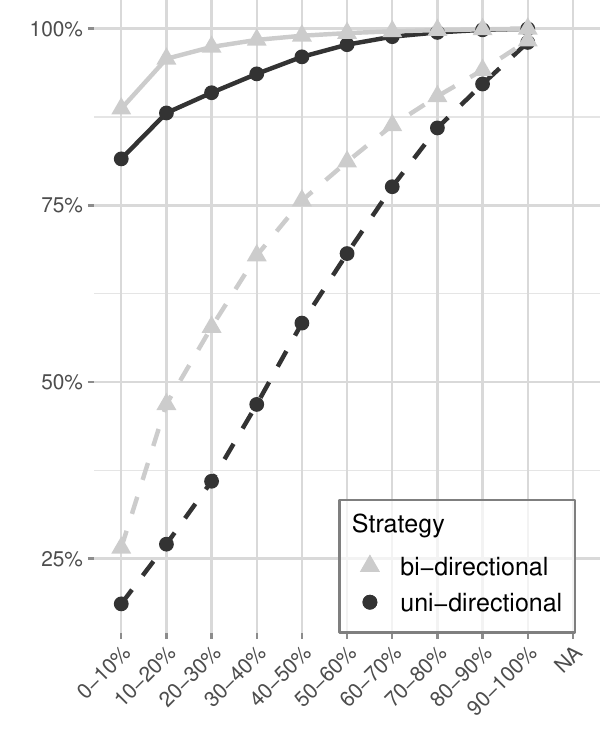}
        \caption{Comparing bi-directional and uni-directional methods. Node-selection strategy fixed to "alternating".}\label{fig:emp-res-bi-dir}
    \end{subfigure}
    \begin{subfigure}[t]{0.47\linewidth}
        \includegraphics[width=0.95\textwidth]{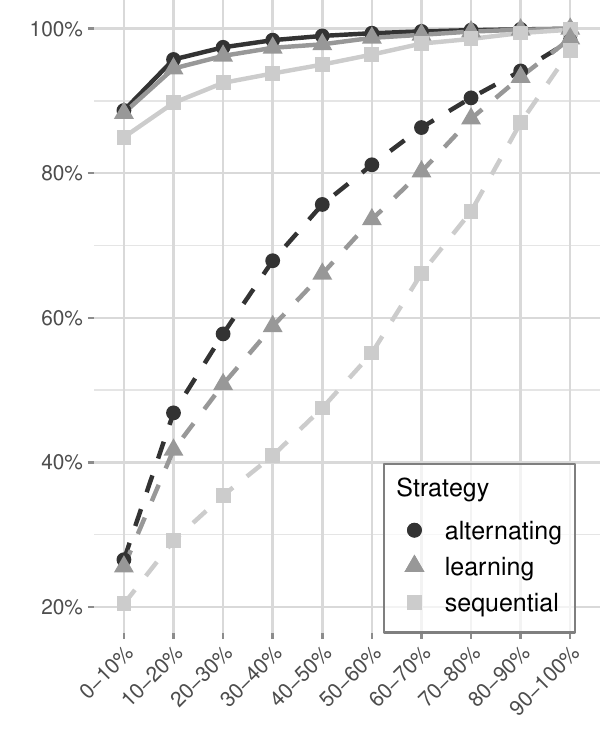}
        \caption{Comparing sub-problem selection strategies. All for the bi-directional strategy.}\label{fig:emp-res-sub-problem-choice}
    \end{subfigure}
      \caption{
      The $x$-axis shows the total number of sub-problem points found, as a proportion of the total number of sub-problem points, binned into intervals. The $y$-axis shows the relative Hypervolume (solid) and the proportion of ROI points found (dashed).%
      }\label{fig:emp2}
    \end{figure}

    We are interested in evaluating the anytime performance of the two configurations of the algorithm: The `uni-directional` method solves sub-problems using the $\epsilon$-constraint method starting from the lower-right part of the Pareto Front, ending with the upper-left point. The `bi-directional` approach described in Algorithm~\ref{ALG:Solving_ROI}, on the other hand, starts from the `centre` of the sub-problem and moves outwards. %
    To evaluate the effect of using the proposed \emph{bi-directional} search for \solveRoi we fix the sub-problem selection strategy to `alternating` in \autoref{fig:emp-res-sub-problem-choice}. 
    
    Both methods return $\Yr$ if no early stopping condition is reached. 
    To compare the convergence of the algorithms towards the set $\Yr$ we will consider an early stopping condition by limiting the number of iterations. In each iteration a sub-problem is solved and, therefore, a new sub-problem solution is generated — unless the sub-problem is infeasible, or the optimal solutions was found when running the \findRoiDecLook. We will consider the \emph{proportion of sub-problem points found} as a normalised measure of the progress. The proportion starts at a level dependent on the number of solutions found by the \findRoiDecLook and the main loop of \solveRoi terminates only when the proportion is $100\%$. 
    We will consider two measures for evaluating the quality of the returned set $\hYr$ as an approximation of $\Yr$. The \emph{Hypervolume} of a set $\Y \subset\R^2$, denoted $\HVoperator(\Y)$ measures the total area of some restricted box, which is dominated by the set $\Y$. 
    The Hypervolume of a set is a widely used measure for the representation quality of \ND sets, and we refer to \citet{guerreiro2021hypervolume} for a discussion of the properties of the measure as well as the computation thereof.
    As a first measure, we will use $\HVRoperator(\hYr)=\frac{\HVoperator(\hat\Yr)}{\HVoperator(\Yr)}$, the relative Hypervolume of $\hYr$ inside the \roi. As a second quality measure, we will count the proportion of interesting solutions found $\PROPoperator(\hYr) = \frac{|\hYr \cap  \Yr|}{|\Yr|}$, as done in \citet{DOMINGUEZRIOS2021210}. In \autoref{fig:emp-res-bi-dir} we compare the error measures $\HVRoperator(\hYr)$ %
    and $\PROPoperator(\hYr)$ for two `bi-directional` as well as a `uni-directional` solution approach. 
    
    The computational results show, that the bi-directional approach outperforms the uni-directional approach, as this consistently have higher $\HVRoperator(\hat\Yr)$ and higher $\PROPoperator(\hYr)$. This confirms the intuition of \solveRoi that, when sub-problem solutions are found starting from the centre followed by an outwards movement, the sub-problem solutions are more likely to combine into solutions in the \roi, as opposed the `uni-directional` approach. 
    
    We conclude that bi-directional search is superior and fix this in the following analysis where we investigate the effects of node selection strategies.

    \subsubsection{Node selection strategies}
    
    We will now shift the focus to testing different strategies for selecting the sub-problem, $s$, and the direction pairs $(s,d)$ (see \autoref{ALG:Solving_ROI}, \autoref{alg-line:node-selection}).
    We consider three selection strategies: alternating, sequential, and a learning-based approach.
    The alternating strategy chooses a new (unsolved) sub-problem in each iteration.
    The sequential strategy solves sub-problems in sequence, starting by finding all solutions to the first sub-problem before proceeding to the second, and so on. Both the alternating and the sequential method look for solutions to sub-problems using the `bi-directional` search.
    The learning strategy is based on the upper-confidence bound approach from reinforcement learning \citep{barto2021reinforcement}, which attempts to balance exploration (trying other sub-problems which have not been picked often) and exploitation (picking sub-problems, which show historical potential of improving Hypervolume). As rewards we use the improvement in relative Hypervolume (inside the region of interest). Using this, we can model the node-selection problem as a restless/non-stationary bandit problem \citep{whittle1988restless}. It is restless because the rewards change over time — one sub-problem might have many solutions that initially improve the relative Hypervolume a lot, but after that maybe no improvements can be gained by finding more solutions to the given sub-problem. 

    An initial small scale study was performed comparing several different learning based node-selection strategies and several configurations of hyperparameters for these. 
    The initial study led to the selection of the upper-confidence-bound approach for the choice for the learning based method.
    In \autoref{fig:emp-res-sub-problem-choice} we compare the results for the node-selection strategies `alternating', `sequential', and `learning'.
     
    We find that the strategy `alternating' performs best, while `learning' is competitive. The strategy defined by `sequential' is on the other hand performing the worst, indicating that it makes sense to alternate between choice of sub-problems. We do not suggest that the `learning'-strategy is the best possible learning-based strategy or that more elaborate search schemes cannot be devised. However, we would like to emphasize that the results suggests that a simple `alternating' strategy is a good heuristic. The mean (median) number of iterations for the algorithms were \nr[204] (\nr[198]). To fully utilize learning-based methods it might be necessary to consider instances requiring many more iterations. Furthermore, we find a clear improvement in convergence when using the `alternating' as opposed to the `sequential' strategy. This is expected, since for a fixed number of known sub-problem points, combining solutions from several evenly sized sets (the alternating case) results into many more solutions, as opposed to the sequential case, where some sets contain almost all of the points, resulting in fewer combinations.

\section{Conclusion}

In this paper, we have proposed a multi-objective perspective on block-structured integer programming problems featuring soft constraints. Specifically, we interpret the soft constraint as a second objective to be optimized, and in doing so, transform the coupled problem into an additively-separable bi-objective optimization problem.
Our multi–objective approach does not make the problem easier as such, since a set of efficient solutions is required instead of a single optimal solution. However, this approach makes the problem decomposable, and provides a decision–maker with trade–off information.
Recognizing that a decision maker is typically only interested in trade-offs close to the soft constraint boundary ($W$), 
we require only trade-off solutions inside a \emph{region of interest}. This simplifies the problem while still providing trade–off solutions, as well as provable bounds on the optimal solution to the single-objective problem.

First, on a theoretical level, we established a new strict upper bound on the number of extreme supported points ($\Yse$) in the multi-objective counterpart problem. By analyzing critical weights, we proved that the total number of global extreme supported points is bounded linearly by the sum of the cardinalities of the sub-problems' extreme supported points. This structural property provides a guarantee against exponentially many extreme supported solutions and thereby bounds the worst-case performance of the dichotomic search procedure.

Next, we developed the \findRoi algorithm for calculating the points ($y^+$ and $y^-$) defining the \emph{region of interest}. Rather than executing an expensive, full \phaseone (dichotomic search), \findRoi performs a binary search in the weight set to dramatically reduce the number of required single-objective integer programming (IP) calls. 
Under strict assumptions, \findRoi requires at most $(3+\log_2(|\Yse|))$ instead of $(2|\Yse|-1)$ IP calls. This was also empirically observed even without the strict assumptions.
Additionally, we introduced the \findRoiDecLook algorithm, which utilizes the decomposable structure, and establishes a $\lambda$-lookup mechanism that uses the optimality of already found solutions for certain $\lambda$-values to reduce the number of IP calls. For large instances the method worked as a strong heuristic providing bounds on the gap to an optimal solution, providing trade-off information in addition to the solution. In several instances, the method found an optimal solution to the single objective problem.

Furthermore, for scenarios where the complete set of \nd points within the ROI is required, we introduced \solveRoi, a bi-directional $\epsilon$-constraint method that searches from the centres of sub-problems outward. Paired with an alternating sub-problem node-selection strategy, this approach yields superior representation convergence compared to simple uni-directional or sequential methods. 
Several promising areas for future research emerge from this work. 
First, the learning-based node-selection strategy, although simple, showed promising anytime performance. Refining the reinforcement learning heuristics could lead to higher efficiency for complex problems requiring large numbers of iterations. 
Second, expanding the multi-objectivization framework to block-structured integer programs with multiple soft coupling constraints represents a natural next step. 
Finally, investigating the effect of the algorithmic set-up on other classes of block-structured problems, such as for instance quadratically coupled problems, is of interest for further research.

\appendix

  \bibliographystyle{elsarticle-harv} 
  \bibliography{literature}

\begin{thebibliography}{42}
\expandafter\ifx\csname natexlab\endcsname\relax\def\natexlab#1{#1}\fi
\providecommand{\url}[1]{\texttt{#1}}
\providecommand{\href}[2]{#2}
\providecommand{\path}[1]{#1}
\providecommand{\DOIprefix}{doi:}
\providecommand{\ArXivprefix}{arXiv:}
\providecommand{\URLprefix}{URL: }
\providecommand{\Pubmedprefix}{pmid:}
\providecommand{\doi}[1]{\href{http://dx.doi.org/#1}{\path{#1}}}
\providecommand{\Pubmed}[1]{\href{pmid:#1}{\path{#1}}}
\providecommand{\bibinfo}[2]{#2}
\ifx\xfnm\relax \def\xfnm[#1]{\unskip,\space#1}\fi
\bibitem[{Aneja and Nair(1979)}]{aneja1979bicriteria}
\bibinfo{author}{Aneja, Y.P.}, \bibinfo{author}{Nair, K.P.},
  \bibinfo{year}{1979}.
\newblock \bibinfo{title}{Bicriteria transportation problem}.
\newblock \bibinfo{journal}{Management Science} \bibinfo{volume}{25},
  \bibinfo{pages}{73--78}.
\bibitem[{Barto(2021)}]{barto2021reinforcement}
\bibinfo{author}{Barto, A.G.}, \bibinfo{year}{2021}.
\newblock \bibinfo{title}{Reinforcement learning: An introduction. by
  richard’s sutton}.
\newblock \bibinfo{journal}{SIAM Rev} \bibinfo{volume}{6},
  \bibinfo{pages}{423}.
\bibitem[{Benders(1962)}]{benders1962partitioning}
\bibinfo{author}{Benders, J.F.}, \bibinfo{year}{1962}.
\newblock \bibinfo{title}{Partitioning procedures for solving mixed-variables
  programming problems}.
\newblock \bibinfo{journal}{Numerische mathematik} \bibinfo{volume}{4},
  \bibinfo{pages}{238–--252}.
\bibitem[{Benson and Sun(2002)}]{benson2002weight}
\bibinfo{author}{Benson, H.P.}, \bibinfo{author}{Sun, E.},
  \bibinfo{year}{2002}.
\newblock \bibinfo{title}{A weight set decomposition algorithm for finding all
  efficient extreme points in the outcome set of a multiple objective linear
  program}.
\newblock \bibinfo{journal}{European Journal of Operational Research}
  \bibinfo{volume}{139}, \bibinfo{pages}{26--41}.
\bibitem[{B{\"o}kler and Mutzel(2015)}]{bokler2015output}
\bibinfo{author}{B{\"o}kler, F.}, \bibinfo{author}{Mutzel, P.},
  \bibinfo{year}{2015}.
\newblock \bibinfo{title}{Output-sensitive algorithms for enumerating the
  extreme nondominated points of multiobjective combinatorial optimization
  problems}, in: \bibinfo{booktitle}{Algorithms-ESA 2015: 23rd Annual European
  Symposium , Patras, Greece, September 14-16, 2015, Proceedings},
  \bibinfo{organization}{Springer}. pp. \bibinfo{pages}{288--299}.
\bibitem[{Boland et~al.(2015)Boland, Charkhgard and Savelsbergh}]{boland2015}
\bibinfo{author}{Boland, N.}, \bibinfo{author}{Charkhgard, H.},
  \bibinfo{author}{Savelsbergh, M.}, \bibinfo{year}{2015}.
\newblock \bibinfo{title}{A criterion space search algorithm for biobjective
  integer programming: The balanced box method}.
\newblock \bibinfo{journal}{INFORMS Journal on Computing} \bibinfo{volume}{27},
  \bibinfo{pages}{735--754}.
\bibitem[{Chalmet et~al.(1986)Chalmet, Lemonidis and
  Elzinga}]{chalmet1986algorithm}
\bibinfo{author}{Chalmet, L.}, \bibinfo{author}{Lemonidis, L.},
  \bibinfo{author}{Elzinga, D.}, \bibinfo{year}{1986}.
\newblock \bibinfo{title}{An algorithm for the bi-criterion integer programming
  problem}.
\newblock \bibinfo{journal}{European Journal of Operational Research}
  \bibinfo{volume}{25}, \bibinfo{pages}{292--300}.
\bibitem[{Chen(2019)}]{chen2019block}
\bibinfo{author}{Chen, L.}, \bibinfo{year}{2019}.
\newblock \bibinfo{title}{On block-structured integer programming and its
  applications}, in: \bibinfo{booktitle}{Nonlinear Combinatorial Optimization}.
  \bibinfo{publisher}{Springer}, pp. \bibinfo{pages}{153--177}.
\bibitem[{Cohon(2013)}]{cohon2013multiobjective}
\bibinfo{author}{Cohon, J.L.}, \bibinfo{year}{2013}.
\newblock \bibinfo{title}{Multiobjective programming and planning}.
\newblock \bibinfo{publisher}{Courier Corporation}.
\bibitem[{Cslovjecsek et~al.(2025)Cslovjecsek, Kouteck{\`y}, Lassota, Pilipczuk
  and Polak}]{cslovjecsek2025}
\bibinfo{author}{Cslovjecsek, J.}, \bibinfo{author}{Kouteck{\`y}, M.},
  \bibinfo{author}{Lassota, A.}, \bibinfo{author}{Pilipczuk, M.},
  \bibinfo{author}{Polak, A.}, \bibinfo{year}{2025}.
\newblock \bibinfo{title}{Parameterized algorithms for block-structured integer
  programs with large entries}.
\newblock \bibinfo{journal}{TheoretiCS} \bibinfo{volume}{4}.
\bibitem[{Dantzig and Wolfe(1960)}]{dantzig1960decomposition}
\bibinfo{author}{Dantzig, G.B.}, \bibinfo{author}{Wolfe, P.},
  \bibinfo{year}{1960}.
\newblock \bibinfo{title}{Decomposition principle for linear programs}.
\newblock \bibinfo{journal}{Operations research} \bibinfo{volume}{8},
  \bibinfo{pages}{101--111}.
\bibitem[{Desrosiers et~al.(2024)Desrosiers, L{\"u}bbecke, Desaulniers and
  Gauthier}]{desrosiers2024branch}
\bibinfo{author}{Desrosiers, J.}, \bibinfo{author}{L{\"u}bbecke, M.},
  \bibinfo{author}{Desaulniers, G.}, \bibinfo{author}{Gauthier, J.B.},
  \bibinfo{year}{2024}.
\newblock \bibinfo{title}{Branch-and-price}.
\newblock \bibinfo{publisher}{Springer}.
\bibitem[{Ángel Domínguez-Ríos et~al.(2021)Ángel Domínguez-Ríos, Chicano
  and Alba}]{DOMINGUEZRIOS2021210}
\bibinfo{author}{Ángel Domínguez-Ríos, M.}, \bibinfo{author}{Chicano, F.},
  \bibinfo{author}{Alba, E.}, \bibinfo{year}{2021}.
\newblock \bibinfo{title}{Effective anytime algorithm for multiobjective
  combinatorial optimization problems}.
\newblock \bibinfo{journal}{Information Sciences} \bibinfo{volume}{565},
  \bibinfo{pages}{210--228}.
\newblock \URLprefix
  \url{https://www.sciencedirect.com/science/article/pii/S0020025521002310},
  \DOIprefix\doi{https://doi.org/10.1016/j.ins.2021.02.074}.
\bibitem[{Drake(2015)}]{drake2015benchmark}
\bibinfo{author}{Drake, J.}, \bibinfo{year}{2015}.
\newblock \bibinfo{title}{Benchmark instances for the multidimensional knapsack
  problem}.
\newblock \bibinfo{journal}{Available from ResearchGate} \bibinfo{volume}{2}.
\newblock \URLprefix
  \url{www.researchgate.net/publication/271198281_Benchmark_instances_for_the_Multidimensional_Knapsack_Problem.},
  \DOIprefix\doi{10.13140/2.1.3578.9122}.
\bibitem[{Ehrgott(2005)}]{ehrgott2005multicriteria}
\bibinfo{author}{Ehrgott, M.}, \bibinfo{year}{2005}.
\newblock \bibinfo{title}{Multicriteria optimization}. volume
  \bibinfo{volume}{491}.
\newblock \bibinfo{publisher}{Springer Science \& Business Media}.
\bibitem[{Eisenbrand et~al.(2018)Eisenbrand, Hunkenschr{\"o}der and
  Klein}]{eisenbrand2018faster}
\bibinfo{author}{Eisenbrand, F.}, \bibinfo{author}{Hunkenschr{\"o}der, C.},
  \bibinfo{author}{Klein, K.M.}, \bibinfo{year}{2018}.
\newblock \bibinfo{title}{Faster algorithms for integer programs with block
  structure}.
\newblock \bibinfo{journal}{arXiv preprint arXiv:1802.06289} .
\bibitem[{Forget et~al.(2020)Forget, Gadegaard, Klamroth, Nielsen and
  Przybylski}]{Forget20}
\bibinfo{author}{Forget, N.}, \bibinfo{author}{Gadegaard, S.},
  \bibinfo{author}{Klamroth, K.}, \bibinfo{author}{Nielsen, L.},
  \bibinfo{author}{Przybylski, A.}, \bibinfo{year}{2020}.
\newblock \bibinfo{title}{Branch-and-bound and objective branching with three
  objectives}.
\newblock \bibinfo{type}{Technical Report}. Optimization Online.
\newblock \URLprefix
  \url{http://www.optimization-online.org/DB_FILE/2020/12/8158.pdf}.
\bibitem[{Gardenghi et~al.(2011)Gardenghi, G{\'{o}}mez, Miguel and
  Wiecek}]{Gardenghi2011}
\bibinfo{author}{Gardenghi, M.}, \bibinfo{author}{G{\'{o}}mez, T.},
  \bibinfo{author}{Miguel, F.}, \bibinfo{author}{Wiecek, M.M.},
  \bibinfo{year}{2011}.
\newblock \bibinfo{title}{Algebra of efficient sets for multiobjective complex
  systems}.
\newblock \bibinfo{journal}{Journal of Optimization Theory and Applications}
  \bibinfo{volume}{149}, \bibinfo{pages}{385--410}.
\newblock \DOIprefix\doi{10.1007/s10957-010-9786-y}.
\bibitem[{Guerreiro et~al.(2021)Guerreiro, Fonseca and
  Paquete}]{guerreiro2021hypervolume}
\bibinfo{author}{Guerreiro, A.P.}, \bibinfo{author}{Fonseca, C.M.},
  \bibinfo{author}{Paquete, L.}, \bibinfo{year}{2021}.
\newblock \bibinfo{title}{The hypervolume indicator: Computational problems and
  algorithms}.
\newblock \bibinfo{journal}{ACM Computing Surveys (CSUR)} \bibinfo{volume}{54},
  \bibinfo{pages}{1--42}.
\bibitem[{Halffmann et~al.(2020)Halffmann, Dietz, Przybylski and
  Ruzika}]{halffmann2020inner}
\bibinfo{author}{Halffmann, P.}, \bibinfo{author}{Dietz, T.},
  \bibinfo{author}{Przybylski, A.}, \bibinfo{author}{Ruzika, S.},
  \bibinfo{year}{2020}.
\newblock \bibinfo{title}{An inner approximation method to compute the weight
  set decomposition of a triobjective mixed-integer problem}.
\newblock \bibinfo{journal}{Journal of Global Optimization}
  \bibinfo{volume}{77}, \bibinfo{pages}{715--742}.
\bibitem[{Helfrich et~al.(2024)Helfrich, Prinz and
  Ruzika}]{helfrich2024weighted}
\bibinfo{author}{Helfrich, S.}, \bibinfo{author}{Prinz, K.},
  \bibinfo{author}{Ruzika, S.}, \bibinfo{year}{2024}.
\newblock \bibinfo{title}{The weighted p-norm weight set decomposition for
  multiobjective discrete optimization problems}.
\newblock \bibinfo{journal}{Journal of Optimization Theory and Applications}
  \bibinfo{volume}{202}, \bibinfo{pages}{1187--1216}.
\bibitem[{Kerb{\'{e}}r{\'{e}}n{\`{e}}s(2022)}]{Kerberenes2022phd}
\bibinfo{author}{Kerb{\'{e}}r{\'{e}}n{\`{e}}s, A.}, \bibinfo{year}{2022}.
\newblock \bibinfo{title}{Multiobjective optimization for complex systems}.
\newblock Ph.D. thesis. Universit{\'e} Paris sciences et lettres.
\newblock \URLprefix \url{https://theses.hal.science/tel-03677499v1}.
\bibitem[{Klamroth et~al.(2013)Klamroth, Köbis, Schöbel and
  Tammer}]{klamroth2013uncertain}
\bibinfo{author}{Klamroth, K.}, \bibinfo{author}{Köbis, E.},
  \bibinfo{author}{Schöbel, A.}, \bibinfo{author}{Tammer, C.},
  \bibinfo{year}{2013}.
\newblock \bibinfo{title}{A unified approach for different concepts of
  robustness and stochastic programming via non-linear scalarizing
  functionals}.
\newblock \bibinfo{journal}{Optimization} \bibinfo{volume}{62},
  \bibinfo{pages}{649--671}.
\newblock \DOIprefix\doi{10.1080/02331934.2013.769104}.
\bibitem[{Klamroth et~al.(2024)Klamroth, Lang and Stiglmayr}]{klamroth2024}
\bibinfo{author}{Klamroth, K.}, \bibinfo{author}{Lang, B.},
  \bibinfo{author}{Stiglmayr, M.}, \bibinfo{year}{2024}.
\newblock \bibinfo{title}{Efficient dominance filtering for unions and
  minkowski sums of non-dominated sets}.
\newblock \bibinfo{journal}{Computers \& Operations Research}
  \bibinfo{volume}{163}, \bibinfo{pages}{106506}.
\newblock \DOIprefix\doi{10.1016/j.cor.2023.106506}.
\bibitem[{Knop et~al.(2020)Knop, Kouteck{\`y} and Mnich}]{knop2020}
\bibinfo{author}{Knop, D.}, \bibinfo{author}{Kouteck{\`y}, M.},
  \bibinfo{author}{Mnich, M.}, \bibinfo{year}{2020}.
\newblock \bibinfo{title}{Combinatorial n-fold integer programming and
  applications}.
\newblock \bibinfo{journal}{Mathematical Programming} \bibinfo{volume}{184},
  \bibinfo{pages}{1--34}.
\bibitem[{Könen and Stiglmayr(2025)}]{koenen2025supported}
\bibinfo{author}{Könen, D.}, \bibinfo{author}{Stiglmayr, M.},
  \bibinfo{year}{2025}.
\newblock \bibinfo{title}{On supportedness in multi-objective integer linear
  programming}.
\newblock \bibinfo{journal}{Journal of Multi-Criteria Decision Analysis}
  \bibinfo{volume}{32}, \bibinfo{pages}{e70024}.
\newblock \DOIprefix\doi{10.1002/mcda.70024}.
\bibitem[{Lyngesen(2026a)}]{bsip-Lyngesen26}
\bibinfo{author}{Lyngesen, M.}, \bibinfo{year}{2026}a.
\newblock \bibinfo{title}{Block-structured integer program {BSIP})}.
\newblock \URLprefix \url{https://github.com/lyngesen/bsip}.
  \bibinfo{note}{{P}ython implementation. accessed 2025-05-02}.
\bibitem[{Lyngesen(2026b)}]{BSIPInstances}
\bibinfo{author}{Lyngesen, M.}, \bibinfo{year}{2026}b.
\newblock \bibinfo{title}{Block-structured integer program instances ({
  MOrepo-Lyngesen26a})}.
\newblock \URLprefix \url{https://github.com/MCDMSociety/MOrepo-Lyngesen26a}.
  \bibinfo{note}{{I}nstance and result files at MOrepo. accessed 2025-05-02}.
\bibitem[{Lyngesen et~al.(2025)Lyngesen, Gadegaard and Nielsen}]{paper1}
\bibinfo{author}{Lyngesen, M.}, \bibinfo{author}{Gadegaard, S.L.},
  \bibinfo{author}{Nielsen, L.R.}, \bibinfo{year}{2025}.
\newblock \bibinfo{title}{Generator sets for the minkowski sum problem}.
\newblock \bibinfo{journal}{European Journal of Operational Research} .
\bibitem[{Martin(2012)}]{martin2012large}
\bibinfo{author}{Martin, R.K.}, \bibinfo{year}{2012}.
\newblock \bibinfo{title}{Large scale linear and integer optimization: a
  unified approach}.
\newblock \bibinfo{publisher}{Springer Science \& Business Media}.
\bibitem[{Przybylski et~al.(2010a)Przybylski, Gandibleux and
  Ehrgott}]{Przybylski2010}
\bibinfo{author}{Przybylski, A.}, \bibinfo{author}{Gandibleux, X.},
  \bibinfo{author}{Ehrgott, M.}, \bibinfo{year}{2010}a.
\newblock \bibinfo{title}{A recursive algorithm for finding all nondominated
  extreme points in the outcome set of a multiobjective integer programme.}
\newblock \bibinfo{journal}{INFORMS Journal on Computing} \bibinfo{volume}{22},
  \bibinfo{pages}{371--386}.
\newblock \URLprefix
  \url{https://research.ebsco.com/linkprocessor/plink?id=3f63fd46-f848-3395-8a2e-0133f5203cd3}.
\bibitem[{Przybylski et~al.(2010b)Przybylski, Gandibleux and
  Ehrgott}]{przybylski2010recursive}
\bibinfo{author}{Przybylski, A.}, \bibinfo{author}{Gandibleux, X.},
  \bibinfo{author}{Ehrgott, M.}, \bibinfo{year}{2010}b.
\newblock \bibinfo{title}{A recursive algorithm for finding all nondominated
  extreme points in the outcome set of a multiobjective integer programme}.
\newblock \bibinfo{journal}{INFORMS Journal on Computing} \bibinfo{volume}{22},
  \bibinfo{pages}{371--386}.
\bibitem[{Rostami et~al.(2017)Rostami, Neri and
  Epitropakis}]{rostami2017progressive}
\bibinfo{author}{Rostami, S.}, \bibinfo{author}{Neri, F.},
  \bibinfo{author}{Epitropakis, M.}, \bibinfo{year}{2017}.
\newblock \bibinfo{title}{Progressive preference articulation for decision
  making in multi-objective optimisation problems}.
\newblock \bibinfo{journal}{Integrated Computer-Aided Engineering}
  \bibinfo{volume}{24}, \bibinfo{pages}{315--335}.
\bibitem[{Schulze(2017)}]{schulze2017new}
\bibinfo{author}{Schulze, B.}, \bibinfo{year}{2017}.
\newblock \bibinfo{title}{New perspectives on multi-objective knapsack
  problems}.
\newblock Ph.D. thesis. Dissertation, Wuppertal, Universit{\"a}t Wuppertal,
  2017.
\bibitem[{Schulze et~al.(2017)Schulze, Paquete, Klamroth and
  Figueira}]{schulze2017soft}
\bibinfo{author}{Schulze, B.}, \bibinfo{author}{Paquete, L.},
  \bibinfo{author}{Klamroth, K.}, \bibinfo{author}{Figueira, J.R.},
  \bibinfo{year}{2017}.
\newblock \bibinfo{title}{Bi-dimensional knapsack problems with one soft
  constraint}.
\newblock \bibinfo{journal}{Comput. Oper. Res.} \bibinfo{volume}{78},
  \bibinfo{pages}{15--26}.
\newblock \DOIprefix\doi{10.1016/j.cor.2016.07.012}.
\bibitem[{Segura et~al.(2016)Segura, Coello, Miranda and Le{
  \'o}n}]{segura2016using}
\bibinfo{author}{Segura, C.}, \bibinfo{author}{Coello, C.A.C.},
  \bibinfo{author}{Miranda, G.}, \bibinfo{author}{Le{ \'o}n, C.},
  \bibinfo{year}{2016}.
\newblock \bibinfo{title}{Using multi-objective evolutionary algorithms for
  single-objective constrained and unconstrained optimization}.
\newblock \bibinfo{journal}{Annals of Operations Research}
  \bibinfo{volume}{240}, \bibinfo{pages}{217--250}.
\bibitem[{Shao and Ehrgott(2014)}]{shao2014objective}
\bibinfo{author}{Shao, L.}, \bibinfo{author}{Ehrgott, M.},
  \bibinfo{year}{2014}.
\newblock \bibinfo{title}{An objective space cut and bound algorithm for convex
  multiplicative programmes}.
\newblock \bibinfo{journal}{Journal of Global Optimization}
  \bibinfo{volume}{58}, \bibinfo{pages}{711--728}.
\bibitem[{Shao and Ehrgott(2016)}]{shao2016primal}
\bibinfo{author}{Shao, L.}, \bibinfo{author}{Ehrgott, M.},
  \bibinfo{year}{2016}.
\newblock \bibinfo{title}{Primal and dual multi-objective linear programming
  algorithms for linear multiplicative programmes}.
\newblock \bibinfo{journal}{Optimization} \bibinfo{volume}{65},
  \bibinfo{pages}{415--431}.
\bibitem[{Wang et~al.(2017)Wang, Olhofer and Jin}]{wang2017}
\bibinfo{author}{Wang, H.}, \bibinfo{author}{Olhofer, M.},
  \bibinfo{author}{Jin, Y.}, \bibinfo{year}{2017}.
\newblock \bibinfo{title}{A mini-review on preference modeling and articulation
  in multi-objective optimization: current status and challenges}.
\newblock \bibinfo{journal}{Complex \& Intelligent Systems}
  \bibinfo{volume}{3}, \bibinfo{pages}{233--245}.
\bibitem[{Whittle(1988)}]{whittle1988restless}
\bibinfo{author}{Whittle, P.}, \bibinfo{year}{1988}.
\newblock \bibinfo{title}{Restless bandits: Activity allocation in a changing
  world}.
\newblock \bibinfo{journal}{Journal of applied probability}
  \bibinfo{volume}{25}, \bibinfo{pages}{287--298}.
\bibitem[{Yu et~al.(2025)Yu, Ma, Wang, Du, Du and Jin}]{yu2025towards}
\bibinfo{author}{Yu, G.}, \bibinfo{author}{Ma, L.}, \bibinfo{author}{Wang, X.},
  \bibinfo{author}{Du, W.}, \bibinfo{author}{Du, W.}, \bibinfo{author}{Jin,
  Y.}, \bibinfo{year}{2025}.
\newblock \bibinfo{title}{Towards fairness-aware multi-objective optimization}.
\newblock \bibinfo{journal}{Complex \& Intelligent Systems}
  \bibinfo{volume}{11}, \bibinfo{pages}{50}.
\bibitem[{Zhou et~al.(2023)Zhou, Du and Arai}]{zhou2023efficient}
\bibinfo{author}{Zhou, D.}, \bibinfo{author}{Du, J.}, \bibinfo{author}{Arai,
  S.}, \bibinfo{year}{2023}.
\newblock \bibinfo{title}{Efficient search of decision makers’ region of
  interest by using preference directions in multi-objective coevolutionary
  algorithm}.
\newblock \bibinfo{journal}{Swarm and Evolutionary Computation}
  \bibinfo{volume}{81}, \bibinfo{pages}{101349}.

\end{thebibliography}
\end{document}